\documentclass[a4paper, 9pt]{article}
\usepackage{graphics}
\usepackage[dvips]{color}

\usepackage{lscape}
\usepackage{latexsym}
\usepackage{amsmath,verbatim}
\usepackage{graphics}
\usepackage{amsthm}
\usepackage{amssymb}
\usepackage[mathscr]{euscript}
\usepackage{mathalfa}
\usepackage{makeidx}
\usepackage{stmaryrd}
\usepackage{multicol}
\usepackage{helvet}
\usepackage{bm}
\usepackage[all]{xy}
\newfont{\Fr}{eufm10}
\newfont{\Sc}{eusm10}
\newfont{\Bb}{msbm10}
\newfont{\Am}{msam10}
\newfont{\am}{msam7}
\numberwithin{equation}{section}
\newtheorem{thm}{Theorem}[section]
\newtheorem{prop}[thm]{Proposition}
\newtheorem{lem}[thm]{Lemma}
\newtheorem{cor}[thm]{Corollary}
\newtheorem{claim}{Claim}{\bf}{\it}

\newtheorem{fthm}{Theorem}{\bf}{\it}
{\bf}{\it}
\newtheorem{fcor}[fthm]{Corollary}{\bf}{\it}
\theoremstyle{definition}
\newtheorem{defn}[thm]{Definition}

{\bf}{\rm}

\theoremstyle{remark}

{\bf}{\it}

\newtheorem{definition and corollary}[thm]{Definition and Corollary}

{\it}{\rm}
\newcommand{\A}{{\mathbb A}}

\newcommand{\C}{{\mathbb C}}
\newcommand{\cO}{{\mathcal O}}
\newcommand{\bK}{{\mathbb K}}

\newcommand{\Hom}{\mbox{\rm Hom}}

\newcommand{\Sym}{\mbox{\Fr S}}

\newcommand{\g}{\mathfrak{g}}
\newcommand{\gb}{\mathfrak{b}}

\newcommand{\h}{\mathfrak{h}}
\newcommand{\gI}{\mathfrak{I}}

\newcommand{\gp}{\mathfrak{p}}
\newcommand{\gu}{\mathfrak{u}}

\renewcommand{\P}{\mathbb{P}}

\newcommand{\sQ}{\mathscr{Q}}
\newcommand{\Q}{\mathbb{Q}}
\newcommand{\R}{\mathbb{R}}

\newcommand{\Z}{\mathbb{Z}}

\newcommand{\Gm}{\mathbb G_m}

\title{Demazure character formula for semi-infinite flag varieties}
\author{Syu \textsc{Kato} \footnote{Department of Mathematics, Kyoto University, Oiwake Kita-Shirakawa Sakyo Kyoto 606-8502 JAPAN \tt{E-mail:syuchan@math.kyoto-u.ac.jp}} \footnote{Research supported in part by JSPS Grant-in-Aid for Scientific Research (B) 26287004 and Kyoto University Jung-Mung program}\footnote{This is a post-peer-review, pre-copyedit version of an article published in {\it Mathematische Annalen}. The final authenticated version is available online at: http://dx.doi.org/10.1007/s00208-018-1652-5}}

\begin{document}
\maketitle

\begin{abstract}
We prove that every Schubert variety of a semi-infinite flag variety is projectively normal. This gives us an interpretation of a Demazure module of a global Weyl module of a current Lie algebra as the (dual) space of global sections of a line bundle on a semi-infinite Schubert variety. Moreover, we give geometric realizations of Feigin-Makedonskyi's generalized Weyl modules, and the $t = \infty$ specialization of non-symmetric Macdonald polynomials.
\end{abstract}

\section*{Introduction}
Semi-infinite flag variety is a variant of affine flag variety that encodes representation theory of affine Lie algebras \cite{FF}. It also admits an interpretation as the space of rational maps, and therefore plays a role in the computation of quantum $K$-theory of flag varieties. This latter direction was pursued by a series of papers by Braverman-Finkelberg \cite{BF14a,BF14b,BF14c}, that leads to the proof of fundamental properties such as normality, rationality of its singularities, an analogue of the Borel-Weil theorem, the computation of quantum $J$-functions (extending the work of Givental-Lee \cite{GL03}), and its connection with $q$-Whittaker functions.

The aim of this paper is two-fold: one is to extend Braverman-Finkelberg's cohomology formula of line bundles to include some naturally twisted sheaves, and the other is to generalize their results to all Schubert varieties so that the situation becomes more satisfactory from representation-theoretic view-points. It turns out that such an extension provides a natural realization of certain specializations of non-symmetric Macdonald polynomials, together with difference equations characterizing them, generalizing their links to the representation theory of current algebras as discovered by Braverman-Finkelberg \cite{BF14b}, Lenart-Naito-Sagaki-Schilling-Shimozono \cite{LNSSS1, LNSSS2, LNSSS3}, Cherednik-Orr \cite{CO}, Naito-Nomoto-Sagaki \cite{NNS}, and Feigin-Makedonskyi \cite{FM}.

To explain what we mean by this, we introduce more notation: Let $G$ be a simply-connected simple algebraic group, let $W$ be its Weyl group with the set $\{s_i\}_{i \in \mathtt I}$ of simple reflections, let $\Lambda$ be the weight lattice, and let $\Lambda _+$ be the set of dominant weights. Let $Q^{\vee}$ be the coroot lattice of $G$. Then, we have the space $\sQ$ of rational maps from $\P^1$ to $G / B$, and its subspace $\sQ ( w )$ formed as the closure of the set of rational maps whose value at $0$ lands on a Schubert variety corresponding to $w \in W$. They carry a natural line bundle $\cO ( \lambda )$ corresponding to each $\lambda \in \Lambda$. Associated to $G$, we have a current algebra $\g [z] := \mathrm{Lie} \, G \otimes _{\C} \C [z]$ and its Iwahori subalgebra $\gI$. The Lie algebra $\g [z]$ also possesses a natural representation $W ( \lambda )$ for each $\lambda \in \Lambda_+$, that is called a global Weyl module (we set $W ( \lambda ) := \{ 0 \}$ if $\lambda \in \Lambda \backslash \Lambda_+$). Kashiwara \cite{Kas05} defined its Demazure submodule $W ( \lambda )_w$ to be the cyclic $\gI$-submodule generated by a vector with weight $w \lambda \in \Lambda$ for each $w \in W$. As they are graded, we have their character $\mathrm{ch} \, W ( \lambda )_w$, valued in $\C (\!(q)\!) [\Lambda]$.

\begin{fthm}[$\doteq$ Theorem \ref{main} $+$ Theorem \ref{dmain}]\label{fmain} For each $\lambda \in \Lambda$ and $w \in W$, we have:
\begin{enumerate}
\item The indscheme $\sQ ( w )$ is normal, and projectively normal;
\item We have the following isomorphism as $\gI$-modules:
$$H^i ( \sQ ( w ), \mathcal O _{\sQ ( w )} ( \lambda ) ) ^* \cong \begin{cases} W ( \lambda )_w & (i=0, \lambda \in \Lambda_+)\\ \{0\} & (\text{otherwise}) \end{cases};$$
\item For each $i \in \mathtt I$ so that $s_i w > w$, we have $\mathrm{ch} \, W ( \lambda )_{s_i w} = D_i ( \mathrm{ch} \,  W ( \lambda )_w)$, where $D_{i}$ is a Demazure operator acting on $\C (\!(q)\!) [\Lambda]$;
\item We also have a Demazure operator $D_{t_{\beta}}$ for each $\beta \in Q^{\vee}$ so that $\left< \beta, w \alpha \right> > 0$ for every positive root $\alpha$, that are mutually commutative. We have
\begin{equation}
D_{t_{\beta}} ( \mathrm{ch} \,  W ( \lambda )_{w} ) = q ^{- \left< \beta, w \lambda \right>} \cdot \mathrm{ch} \,  W ( \lambda )_{w}.\label{fdif}
\end{equation}
\end{enumerate}
\end{fthm}

We remark that Theorem \ref{fmain} 2)--4) can be regarded as a semi-infinite analogue of the Demazure character formula due to Demazure-Joseph-Kumar in the ordinary setting (\cite{Dem74,Jos84}, see Kumar \cite{Kum02} V\!I\!I\!I), that contains difference equations (\ref{fdif}) characterizing them.

\begin{fthm}[$=$ Theorem \ref{FM} $+$ Corollary \ref{BW}]\label{ffree}
For each $w \in W$ and $\lambda \in \Lambda _+$, the module $W ( \lambda )_w$ admits a free action of a certain polynomial ring and its specialization to $\C$ gives the Feigin-Makedonskyi module $W _{w \lambda}$. In particular, we have
$$\Gamma ( \mathrm{Fl} ^{\frac{\infty}{2}}_{G} ( w ), \cO_{\mathrm{Fl} ^{\frac{\infty}{2}}_{G} ( w )} ( \lambda ) ) ^* \cong W _{w \lambda},$$
where $\mathrm{Fl} ^{\frac{\infty}{2}}_{G} ( w )$ is a variant of $\sQ ( w )$.
\end{fthm}

Cherednik-Orr \cite{CO} obtained a recursive formula for non-symmetric Macdonald polynomials specialized to $t = \infty$. The comparison with our construction yields:

\begin{fthm}[$=$ Corollary \ref{gnsMac}]\label{fnsMac}
For each $\lambda \in \Lambda_+$ and $w \in W$, there exists an $( \mathbf I \rtimes \Gm )$-equivariant sheaf $\mathcal E_{w} ( \lambda )$ so that
$$\mathrm{ch} \, H^0 ( \sQ ( w ), \mathcal E_{w} ( \lambda ) )^* = \left( \prod_{i \in \mathtt I} \prod_{k = 1} ^{\left< \alpha_i ^{\vee}, \lambda_{w} \right>} \frac{1}{1 - q ^{k}} \right) \cdot E_{- w \lambda} ^{\dagger} ( q^{-1}, \infty ),$$
where $\lambda_w$ is a dominant weight determined by $\lambda$ and $w$, and $E_{- w \lambda}^{\dagger} ( q, t )$ is $($the bar-conjugate of$)$ a non-symmetric Macdonald polynomial $($see $\S 5)$. In addition, we have $H ^i ( \sQ ( w ), \mathcal E _w ( \lambda) ) = \{ 0 \}$ for $i > 0$.
\end{fthm}

We remark that the vector space $H^0 ( \sQ ( w ), \mathcal E_{w} ( \lambda ) )^*$ is a cyclic $\gI$-module (Lemma \ref{restE}). One thing missing here at the moment is an analogue of Theorem \ref{ffree} in the setting of Theorem \ref{fnsMac}.

In the course of its proof, we present an analogue of the Kodaira vanishing theorem (Proposition \ref{coh-tw}) along the line of Kumar \cite{Kum02}. Some particular instances of our results are two formulas, one of which is \cite{CO} Proposition 2.5:

\begin{fcor}[$=$ Corollary \ref{nMconn}]\label{fcor}
For each $\lambda \in \Lambda _+$, we have the following relations between different specializations of non-symmetric Macdonald polynomials:
\begin{align*}
D_{w_0} ( E_{w_0 \lambda} ^{\dagger} ( q^{-1}, \infty ) ) & = E_{w_0 \lambda}^{\dagger} ( q, 0 )\\
D_{w_0 t_{\beta}} ( E_{w_0 \lambda}^{\dagger} ( q, 0 ) ) & = q^{\left< \beta, \lambda \right>} \cdot E_{w_0 \lambda} ^{\dagger} ( q^{-1}, \infty ),
\end{align*}
where $w_0 \in W$ is the longest element, $\beta \in Q^{\vee}$ satisfies $\left< \beta, \alpha_i \right> < 0$ for each $i \in \mathtt I$, and $t_{\beta}$ is the translation element in the affine Weyl group $W \ltimes Q^{\vee}$.
\end{fcor}

The organization of this note is as follows. The first two sections contain preliminary material on current algebra representations and semi-infinite flag variety, respectively. We provide proofs of some facts for which the author was unable to find appropriate references. The third section is a preparatory observation that the semi-infinite flag variety must be actually projectively normal. The fourth section contains a proof of Theorem A through algebraic manipulations. Taking the works of Braverman-Finkelberg \cite{BF14a,BF14b,BF14c} into account, the idea is supported by the fact that the Demazure character formula is in fact equivalent to the normality of Schubert varieties in the classical case. The fifth section contains a proof of Theorem B. Its main argument gives a simple (to the author's point of view) explanation of a result of Feigin-Makedonskyi-Orr \cite{FMO} (cf. Naito-Nomoto-Sagaki \cite{NNS}). The sixth section is about Theorem C, that is a geometric interpretation of the intertwiners in the theory of non-symmetric Macdonald polynomials at $t=\infty$ due to Cherednik-Orr \cite{CO} (which can be also seen as a semi-infinite analogue of the $t = 0$ specialization of non-symmetric Macdonald polynomials obtained by Sanderson and Ion \cite{San00, Ion03}).

After completed the first version of this paper, the author learned that a part of Theorem \ref{fmain} is also formulated in \cite{BFunp} from a slightly different perspective.

{\small
{\bf Acknowledgement:} The author would like to thank Michael Finkelberg for attracting his attention to \cite{FM} and sent me his unpublished note \cite{BFunp}. He also would like to thank Satoshi Naito for various comments and suggestions on the topic presented in this note, Shrawan Kumar for discussion on semi-infinite flag varieties, and Evgeny Feigin and Daniel Orr for preventing him from some incorrect references. The original version of this note was written during the author's stay at MIT in the academic year 2015/2016. The author would like to thank George Lusztig and MIT for their hospitality. Finally, the author would like to express his thanks to the referee who have kindly made many remarks on the previous version of this note.}

\section{Preparatory materials}
Throughout this note, a variety is a separated reduced scheme of finite type over $\mathbb C$, and its points are closed points unless otherwise stated.

A vector space is always a $\C$-vector space, and a graded vector space refers to a $\Z$-graded vector space whose grading is bounded from below and each of its graded piece is finite-dimensional. For a graded vector space $M = \bigoplus_{i \in \Z} M _i$ or its completion $M = \prod_{i \in \Z} M_i$, we define its dual as $M^* := \bigoplus_{i \in \Z} \mathrm{Hom}_{\C} ( M _i, \C )$, where $\mathrm{Hom}_{\C} ( M _i, \C )$ is understood to have degree $-i$. (We sometimes deal with the graded completion of the dual of a graded module, that is not a graded module in our sense. In such an occasion, we regrade the module in an opposite way if necessary.) We define the graded dimension of a graded vector space as
$$\mathsf{gdim} \, M := \sum_{i\in \Z} q^i \dim _{\C} \, M_i \in \Q (\!(q)\!).$$

For each $n, k$, we denote by $\C [\A^{(n)}] _{\le k}$ the degree $\le k$-part of the symmetric polynomial ring in $n$-variables (of their degrees one).

\subsection{Generality}
Let $G$ be a connected, simply connected simple algebraic group over $\C$, and let $B$ and $H$ be a Borel subgroup and a maximal torus of $G$ so that $H \subset B$. We set $U$ $(= [B,B])$ to be the unipotent radical of $B$ and let $U^-$ be the opposite unipotent subgroup of $U$ with respect to $H$. We denote the Lie algebra of an algebraic group by German letters. We have a (finite) Weyl group $W := N_G ( H ) / H$. For an algebraic group $E$, we denote its set of $\C [z]$-valued points by $E [z]$, its set of $\C [\![z]\!]$-valued points by $E [\![z]\!]$, and its set of $\C (z)$-valued points by $E (z)$.

Let $\Lambda := \mathrm{Hom} _{gr} ( H, \C ^{\times} )$ be the weight lattice of $H$, let $\Delta \subset \Lambda$ be the set of roots, let $\Delta_+ \subset \Delta$ be the set of roots that yield root subspaces in $\gb$, and let $\Pi \subset \Delta _+$ be the set of simple roots. We set $\Delta_- := - \Delta_+$. For $\lambda, \mu \in \Lambda$, we define $\lambda \ge \mu$ if and only if $\lambda - \mu \in \Z_{\ge 0} \Delta_+$. Let $Q^{\vee}$ be the dual lattice of $\Lambda$ with a natural pairing $\left< \bullet, \bullet \right> : Q^{\vee} \times \Lambda \rightarrow \Z$. We define $\Pi^{\vee} \subset Q ^{\vee}$ to be the set of positive simple coroots, and let $Q_+^{\vee} \subset Q ^{\vee}$ be the set of non-negative integer span of $\Pi^{\vee}$. We set $\Lambda_+ := \{ \lambda \in \Lambda \mid \left< \alpha^{\vee}, \lambda \right> \ge 0, \hskip 2mm \forall \alpha^{\vee} \in \Pi^{\vee} \}$. Let $r$ be the rank of $G$ and we set $\mathtt I := \{1,2,\ldots,r\}$. We fix bijections $\mathtt I \cong \Pi \cong \Pi^{\vee}$ so that $i \in \mathtt I$ corresponds to $\alpha_i \in \Pi$, its coroot $\alpha_i^{\vee} \in \Pi ^{\vee}$, and a simple reflection $s_i \in W$ corresponding to $\alpha_i$. We also have a reflection $s_{\alpha} \in W$ corresponding to $\alpha \in \Delta_+$. Let $\ell : W \rightarrow \Z_{\ge 0}$ be the length function and let $w_0 \in W$ be the longest element. Let $\Delta_{\mathrm{aff}} := \Delta \times \Z \delta \cup \{m \delta\}_{m \neq 0}$ be the untwisted affine root system of $\Delta$ with its positive part $\Delta_+ \subset \Delta_{\mathrm{aff}, +}$. We set $\alpha_0 := - \vartheta + \delta$, $\Pi_{\mathrm{aff}} := \Pi \cup \{ \alpha_0 \}$, and $\mathtt I_{\mathrm{aff}} := \mathtt I \cup \{ 0 \}$, where $\vartheta$ is the highest root of $\Delta_+$. We set $W _{\mathrm{aff}} := W \ltimes Q^{\vee}$ and call it the affine Weyl group. It is a reflection group generated by $\{s_i \mid i \in \mathtt I_{\mathrm{aff}} \}$, where $s_0$ is the reflection with respect to $\alpha_0$. Sending $s_0 \mapsto s_{\vartheta}$ (and $s_i \mapsto s_i$ for $i \in \mathtt I$) induces a group homomorphism $W_{\mathrm{aff}} \ni w \mapsto \overline{w} \in W$. Together with the normalization $t_{- \vartheta^{\vee}} := s_{\vartheta} s_0$ (for the coroot $\vartheta^{\vee}$ of $\vartheta$), we introduce the translation element $t_{\beta} \in W _{\mathrm{aff}}$ for each $\beta \in Q^{\vee}$.

Let $\mathtt{ev}_0 : G [z] \rightarrow G$ be the evaluation map at $z = 0$. For each $\mathtt J \subset \mathtt I$, we have a Coxeter subgroup $W _{\mathtt J} \subset W$ whose simple reflections are $\{s_i \mid i \in \mathtt J \}$ and a parabolic subgroup $B \subset P_{\mathtt J} \subset G$ whose Weyl group (of the Levi part) is naturally identified with $W_{\mathtt J}$. We set $\mathbf I_{\mathtt J} := \mathtt{ev}_0^{-1} ( P_{\mathtt J} )$. We set $\mathbf I : = \mathbf I_{\emptyset}$ and call it the Iwahori subgroup of $G [z]$. We set $\gI := \mathrm{Lie} \, \mathbf I$ and $\gI_{\mathtt J} := \mathrm{Lie} \, \mathbf I_{\mathtt J}$. We have a unique minimal connected closed subgroup $G [z] \not \supset \mathbf I_0 \subset G ( z )$ that contains $\mathbf I$ ($= \mathbf I_{\emptyset}$). For each $i \in \mathtt I$, we denote by $B^0_i$ the intersection of $\mathbf I$ with the semi-simple Levi component $L_i^0$ of $\mathbf I_i$ that is stable by the adjoint $H$-action.

For each $\lambda \in \Lambda _+$, we denote by $V ( \lambda )$ (or $V _G ( \lambda )$ in case we specify $G$) the irreducible finite-dimensional $\g$-module with its highest weight $\lambda$. It is standard that we have a unique non-zero vector $v_{w \lambda} \in V ( \lambda )$ of weight $w \lambda$ up to scalar for each $w \in W$.

Let $\varpi_1,\ldots,\varpi_r \in \Lambda_+$ be the dual basis of $\Pi^{\vee}$. For $\lambda \in \Lambda_+$, we expand it as
$$\lambda = \sum_{i=1}^r \lambda_i \varpi_i \hskip 3mm \text{with} \hskip 3mm \lambda_i \in \Z_{\ge 0} \hskip 3mm\text{for} \hskip 3mm 1 \le i \le r$$
and define $|\lambda| := \sum _{i=1}^r \lambda_i$ and $\lambda! := \prod_{i=1}^r \lambda_i!$. We also identify $\lambda$ with a composition $(\lambda_1,\ldots, \lambda_r) \in \Z _{\ge 0}^r$. Using this identification, we define
$$\C [ \A ^{(\lambda)} ] := \bigotimes _{i=1}^r \C [x_{i,1},\ldots,x_{i,\lambda_i}] ^{\Sym_{\lambda_i}} \subset \bigotimes _{i=1}^r \C [x_{i,1},\ldots,x_{i,\lambda_i}] =: \C [ \A ^{\lambda} ].$$

Let $\widehat{\g}$ be the untwisted affine Kac-Moody Lie algebra arising from $\g$, and let $\g [z] := \g \otimes _{\C} \C [z]$ be the current algebra of $\g$. We have natural inclusions $\g \subset \g [z] \subset \widehat{\g}$. Let $\widehat{\h} = \h \oplus \C K \oplus \C d \subset \widehat{\g}$ be the Cartan subalgebra that prolongs $\h \subset \g$ with a convention that $[ K , \widehat{\g} ] = 0$ and $d$ is the degree operator of $\g [z]$. We equip a $\Z$-grading of $\g [z]$ by setting $\deg \, \xi \otimes z^m = m$ for every $\xi \in \g \setminus \{ 0 \}$ (this is the grading induced by the $d$-action). We note that $U ( \g [z] )$ is not a graded vector space in our sense.

Let $\bK := \C ( t )$ and let $U_t$ be the quantum loop algebra of $\widehat{\g}$ with its quantum parameter $t$ (see e.g. \cite{Kas05} 2.1). It has the positive part $U_t^+ \subset U_t$, the Cartan part $U_t^0 \subset U_t$, and the classical part $U_t^{\flat} \subset U_t$. We have their $\C [t]$-integral lattices $\mathbf U_t^{\flat}, \mathbf U_t^+, \mathbf U_t^0$ so that
$$\mathbf U_t^+ \otimes_{\C [t]} \C _1 \cong U ( [\gI, \gI] ), \hskip 3mm \mathbf U_t^0 \otimes_{\C [t]} \C _1 \subset U ( \h \oplus \C ) ^{\wedge}, \hskip 1.5mm \text{and} \hskip 1.5mm \mathbf U_t^{\flat} \otimes_{\C [t]} \C _1 \subset U ( \g )^{\wedge},$$
where $U ( \h \oplus \C ) ^{\wedge}$ and $U ( \g )^{\wedge}$ are the integral weight idempotents completions of $U ( \h \oplus \C )$ and $U ( \g )$, respectively, and their inclusions are dense. We set $U_t^{\ge 0} := U_t^{+} U_t^0 \subset U_t$. The algebra $U _t$ also admits an $\exp d$-action (by embedding it into a quantum algebra of Kac-Moody type) that commutes with $U_t^{\flat}$, so that the degree $\exp ( m )$-part of $U_t$ corresponds to the degree $m$-part of $U ( \g [z,z^{-1}] )$ for each $m \in \Z$. We regrade this degree $\exp (m)$-part of $U_t^+$ as the degree $m$-part.

For each $0 \neq \lambda \in \Lambda_+$ and $x \in \C$, we sometimes regard $V ( \lambda )$ as an irreducible $\g [z]$-module via the Lie algebra quotient map $\g [z] \rightarrow \g [z] / ( z - x ) \g [z] \cong \g$, that we denote by $V ( \lambda, x )$. (We note that $V ( 0, x ) = V ( 0, 0 )$ for every $x \in \C$.) For a graded $\gI$-module $M$, we define its character as
$$\mathrm{ch} \, M := \sum_{\lambda \in \Lambda} e ^{\lambda} \, \mathsf{gdim} \, \mathrm{Hom} _{\h} ( \C_{\lambda}, M ) \in \Q (\!(q)\!) [\Lambda].$$
We replace $\h$ with $\bK [ Q^{\vee} ] \subset U_t$ to define a character of a $U_t^{\ge 0}$-module (with the multiplicative action on $\C_{\lambda}$). For two such modules $M$ and $N$, we denote $\mathrm{ch} \, M \le \mathrm{ch} \, N$ if the corresponding inequality holds for every coefficient of $q^k e^{\lambda}$ ($k \in \Z, \lambda \in \Lambda$). Each $V ( \lambda )$ ($\lambda \in \Lambda_+$) admits a lift $V_t ( \lambda )$ into a $U_t^{\flat}$-module so that $\mathrm{ch} \, V ( \lambda, 0 ) = \mathrm{ch} \, V_t ( \lambda )$ by further extending to a $U_t^{\ge 0}$-module concentrated in degree $0$.

Let $X := G / B$ be the flag variety of $G$, that we sometimes denote by $X_G$. For each $\lambda \in \Lambda$, we have a line bundle $G \times ^B \lambda$, that we denote by $\cO _X ( \lambda )$. For each $w \in W$, we have a $B$-orbit $\mathbb O (w) \subset X$ obtained as $B \dot{w} B/B \subset X$ with a unique $H$-fixed point $x_w$, where $\dot{w} \in N_G ( H )$ is a lift of $w$ (so that $\mathbb O ( w )$ is independent of the choice). We set $X ( w ) := \overline{\mathbb O ( w )}$. It is well-known that $\dim \, X ( w ) = \ell ( w )$. For $w, w' \in W$, we write $w > w'$ if and only if $X ( w ) \supset X ( w' )$.

\subsection{Current algebras}\label{CAcat}

\begin{defn}[integrable module]
A $\g [z]$-module $M$ is said to be integrable if and only if $M$ decomposes into a direct sum of finite-dimensional $\g$-modules. Let $\g [z] \mathchar`-\mathsf{mod}$ be the category of finitely generated integrable $\g [z]$-module. For each $\lambda \in \Lambda _+$, let $\g [z] \mathchar`-\mathsf{mod}^{\le \lambda}$ be the fullsubcategory of $\g [z] \mathchar`-\mathsf{mod}$ whose object is isomorphic to a direct sum of $\g$-modules in $\{ V (\mu)\}_{\mu \le \lambda}$.
\end{defn}

\begin{defn}[projective modules and global Weyl module]\label{gWP}
For each $\lambda \in \Lambda _+$, we define the non-restricted projective module $P ( \lambda )$ as
$$P ( \lambda ) := U ( \g [z] ) \otimes _{U ( \g )} V ( \lambda ).$$
Let $P ( \lambda ;\mu )$ be the largest $\g [z]$-module quotient of $P ( \lambda )$ so that
\begin{equation}
  \Hom _{\g} ( V (\gamma), P ( \lambda; \mu ) ) = \{ 0 \} \hskip 5mm \text{if} \hskip 5mm \gamma \not\le \mu.\label{cut}
\end{equation}
We define the global Weyl module $W ( \lambda )$ of $\g$ to be $P ( \lambda; \lambda )$.
\end{defn}

\begin{lem}\label{graded-mods}
The projective module $P ( \lambda )$, its quotient $P ( \lambda; \mu )$ and global Weyl modules $W ( \lambda )$ can be regarded as graded modules with a simple head $V ( \lambda, 0 )$ sitting at degree $0$ $($for $\lambda, \mu \in \Lambda_+)$.
\end{lem}

\begin{proof}
Straight-forward from the construction.
\end{proof}

\begin{thm}[Chari-Loktev \cite{CL06}, Fourier-Littelmann \cite{FL07}, Naoi \cite{Nao12}]\label{free}
For each $\lambda \in \Lambda _+$, it holds:
\begin{enumerate}
\item the module $P ( \lambda )$ is the projective cover of $V ( \lambda, x )$ as an integrable $\g [z]$-module for every $x \in \C$;
\item the module $W ( \lambda )$ admits a free action of $\C [ \A ^{(\lambda)} ]$ induced by the $U ( \h [z] )$-action on the $\h$-weight $\lambda$-part of $W ( \lambda )$, that commutes with the $\g [z]$-action;
\item the natural grading structure of $\C [ \A^{(\lambda)} ]$ respects the grading of $W ( \lambda )$.
\end{enumerate}
We set $\A^{(\lambda)} := \mathrm{Spec} \, \C [\A^{(\lambda)}]$. For each $x \in \A^{(\lambda)}$, we have the specialization $W ( \lambda, x ) := W ( \lambda ) \otimes_{\C [\A^{(\lambda)}]} \C _x$ of $W ( \lambda )$ corresponding to $x$.
\begin{itemize}
\item[4.] $W ( \lambda, x ) \cong W ( \lambda, y )$ as $\g$-modules for each $x, y \in \A^{(\lambda)}$;
\item[5.] if $x \in \A^{(\lambda)}$ is the orbit of $|\lambda|$-distinct points, then we have
$$W ( \lambda, x ) \cong \bigotimes _{i=1}^r \bigotimes _{j=1}^{\lambda_i} W ( \varpi_i, x_{i,j} ).$$
Here $( x_{i,1},\ldots,x_{i,\lambda_i} ) \in \A ^{\lambda_i}$ corresponds to $x$ $($up to $\Sym_{\lambda_i}$-action$)$.
\end{itemize}
\end{thm}

\begin{proof}
The assertion 1) follows by the definition of $P ( \lambda )$ through the Frobenius reciprocity. As explained in Chari-Ion \cite[2.8--2.10]{CI14}, the simply-laced cases of the assertions 2)--5) are contained in \cite{FL07} and the non simply-laced cases are contained in \cite{Nao12}.
\end{proof}

\begin{defn}[local Weyl module]
For each $\lambda \in \Lambda _+$ and $x \in \A ^{(\lambda)}$, we call $W ( \lambda, x )$ (in Theorem \ref{free}) the local Weyl module supported on $x$.
\end{defn}

\begin{thm}[Chari-Loktev, Fourier-Littelmann, Naoi]\label{qwc}
For each $\lambda \in \Lambda_+$, there exists a $U_t^{\ge 0}$-module $W _t ( \lambda )$ with a $\C [t]$-lattice so that its specialization to $t = 1$ yields $W ( \lambda )$. In particular, we have $\mathrm{ch} \, W_t ( \lambda ) = \mathrm{ch} \, W ( \lambda )$.
\end{thm}

\begin{proof}
By Kashiwara \cite{Kas94, Kas05}, we have $W_t ( \lambda )$ defined as a Demazure submodule of an extremal weight module of $U_t$, equipped with a global basis. Hence, we have a graded $\g [z]$-module $W' ( \lambda )$ obtained as the $t = 1$ specialization of (the $\C [t]$-lattice of) $W_t ( \lambda )$ (obtained by global bases). By our definition of $W ( \lambda )$ and \cite{Kas02} Corollary 5.2, we have a natural map $\eta : W ( \lambda ) \to W' ( \lambda )$ of graded $\g [z]$-modules. The comparison of the $U ( \h [z] )$-actions by Theorem \ref{free} 2) and Beck-Nakajima \cite{BN04} Theorem 4.16 implies that $\eta$ induces isomorphism between weight $\lambda$-part. Moreover, Chari-Pressley \cite{CP01} Lemma 4.6 implies that the induced map $W ( \lambda, 0 ) \rightarrow W' ( \lambda ) \otimes _{\C [\A^{(\lambda)}]} \C_0$ is surjective. Hence, we conclude $W ( \lambda, 0 ) \cong W' ( \lambda ) \otimes _{\C [\A^{(\lambda)}]} \C_0$ from the character comparisons in Chari-Loktev \cite{CL06}, Fourier-Littelmann \cite{FL07}, and Naoi \cite{Nao12}. Therefore, (the graded version of) Nakayama's lemma implies that $\eta$ is an isomorphism. This yields all the assertions.

Note that the comparison of Naito-Sagaki \cite{NS14} Theorem 6.4.1 and Chari-Ion \cite{CI14} Proposition 4.3 yields $\mathrm{ch} \, W_t ( \lambda ) = \mathrm{ch} \, W ( \lambda )$ directly.
\end{proof}

\section{Semi-infinite Schubert varieties}
We review the quasi-map realization of semi-infinite flag variety of $G$, for which the basic references are Finkelberg-Mirkovi\'c \cite{FM99} and Feigin-Finkelberg-Kuznetsov-Mirkovi\'c \cite{FFKM}.

We have $W$-equivariant isomorphisms $H^2 ( X, \Z ) \cong \Lambda$ and $H_2 ( X, \Z ) \cong Q ^{\vee}$. This identifies the ample cone of $X$ with $\Lambda_+ \subset \Lambda$ and the effective cone of $X$ with $Q_+^{\vee}$. A quasi-map $( f, D )$ is a map $f : \P ^1 \rightarrow X$ together with a $\Pi$-colored effective divisor
$$D = \sum_{\alpha^{\vee} \in \Pi^{\vee}, x \in \P^1} m_x (\alpha^{\vee}) \alpha^{\vee} \otimes (x) \in Q^{\vee} \otimes_\Z \mathrm{Div} \, \P^1 \hskip 3mm \text{with} \hskip 3mm m_x (\alpha^{\vee}) \in \Z_{\ge 0}.$$
For $i \in \mathtt I$, we set $D_i := \left< D, \varpi_i \right> \in \mathrm{Div} \, \P^1$. We call $D$ the defect of the quasi-map $(f, D)$. Here we define the degree of the defect by
$$|D| := \sum_{\alpha^{\vee} \in \Pi^{\vee}, x \in \P^1} m_x (\alpha^{\vee}) \alpha^{\vee} \in Q_+^{\vee}.$$

\begin{thm}[Drinfeld-Pl\"ucker data over fields, see Braverman-Gaitsgory \cite{BG} 1.1.2]\label{DPf}
Let $\bK$ be an overfield of $\C$. Then, the set of collections $\{ \bK v_{\lambda} \}_{\lambda \in \Lambda_{+}}$ of lines in $V ( \lambda ) \otimes_{\C} \bK$ so that
$$v_{\lambda} \otimes_{\bK} v_{\mu} \in \bK v_{\lambda + \mu} \subset V ( \lambda  + \mu ) \otimes_{\C} \bK \subset V ( \lambda ) \otimes_{\C} V( \mu ) \otimes_{\C} \bK \hskip 3mm \text{for each} \hskip 3mm \lambda, \mu \in \Lambda_{+}$$
is in bijection with the set of closed $\bK$-points of $X$. \hfill $\Box$
\end{thm}

\begin{defn}[Drinfeld-Pl\"ucker data]\label{Zas}
Consider a collection $\mathcal L = \{( \psi_{\lambda}, \mathcal L^{\lambda} ) \}_{\lambda \in \Lambda_+}$ of inclusions $\psi_{\lambda} : \mathcal L ^{\lambda} \hookrightarrow V ( \lambda ) \otimes _{\C} \mathcal O _{\P^1}$ of line bundles $\mathcal L ^{\lambda}$ over $\P^1$. The data $\mathcal L$ is called a Drinfeld-Pl\"ucker data (DP-data) if the canonical inclusion of $G$-modules
$$\eta_{\lambda, \mu} : V ( \lambda + \mu ) \hookrightarrow V ( \lambda ) \otimes V ( \mu )$$
induces an isomorphism
$$\eta_{\lambda, \mu} \otimes \mathrm{id} : \psi_{\lambda + \mu} ( \mathcal L ^{\lambda + \mu} ) \stackrel{\cong}{\longrightarrow} \psi _{\lambda} ( \mathcal L^{\lambda} ) \otimes_{\cO_{\P^1}} \psi_{\mu} ( \mathcal L^{\mu} )$$
for every $\lambda, \mu \in \Lambda_+$.
\end{defn}

For each $\beta \in Q_+^{\vee}$, we set
$$\sQ ( X, \beta ) : = \{ f : \P ^1 \rightarrow X \mid \text{ quasi-map s.t. } f _* [ \P^1 ] + | D | = \beta \},$$
where $f_* [\P^1]$ is the class of the image of $\P^1$ multiplied by the degree of $\P^1 \to \mathrm{Im} \, f$. We sometimes denote $\sQ ( X, \beta )$ by $\sQ ( \beta )$ in case there is no danger of confusion, and also for various varieties and indschemes of the form $\sQ_{?} (X, w, ?)$ defined below. The topology of this space arises from:

\begin{thm}[Drinfeld, see Finkelberg-Mirkovi\'c \cite{FM99}]
The variety $\sQ ( X, \beta )$ is isomorphic to the variety formed by isomorphism classes of the DP-data $\mathcal L = \{( \psi_{\lambda}, \mathcal L^{\lambda} ) \}_{\lambda \in \Lambda_+}$ such that $\deg \, \mathcal L ^{\lambda} = - \left< \beta, \lambda \right>$.
\end{thm}

For each $\beta, \beta' \in Q_+ ^{\vee}$, we have an embedding
$$\imath^{\beta,\beta'} : \sQ ( \beta ) \hookrightarrow \sQ ( \beta + \beta' ),$$
that simply adds the defect by $\beta' \otimes (\infty)$. We set $\sQ (X) := \varinjlim_{\beta} \sQ ( X, \beta )$ and call it the (indscheme model of the) semi-infinite flag variety of $G$. We have a natural $G [z]$-action on $\sQ$ that preserves the defect.

Let $\sQ_0 ( X )$ denote the subspace of $\sQ ( X )$ whose defect is supported outside of $0 \in \P^1$. We have a natural evaluation map
$$\mathtt{ev} _0 : \sQ _{0} \longrightarrow X,$$
that is $G [z]$-equivariant. It restricts to $\sQ_{0} ( \beta ) \subset \sQ ( \beta )$ for each $\beta \in Q_+ ^{\vee}$. For each $w \in W$, we define $\sQ ( X, w ) := \overline{\mathtt{ev}_{0}^{-1} ( X (w) )}$ and call it the semi-infinite Schubert variety.

For each $\lambda \in \Lambda$, we have a $G[z]$-equivariant line bundle $\cO _{\sQ ( \beta )} ( \lambda )$ (and its pro-object $\cO _{\sQ} ( \lambda )$) obtained by the (tensor product of the) pull-backs $\cO _{\sQ ( \beta )}( \varpi_i )$ of the $i$-th $\cO ( 1 )$ via the embedding
\begin{equation}
\sQ ( \beta ) \hookrightarrow \prod_{i \in \mathtt I} \P ( V ( \varpi_i ) \otimes_{\C} \C [z] _{\le \left< \beta, \varpi_i \right>} ),\label{Pemb}
\end{equation}
for each $\beta \in Q_+^{\vee}$ (see e.g. \cite{BF14b} \S 2.1).

We set $\cO _{\sQ ( w )} ( \lambda )$ and $\cO _{\sQ (w, \beta)} ( \lambda )$ ($\beta \in Q_+^{\vee}$) to be the pullback of $\cO _{\sQ} ( \lambda )$ to $\sQ ( w )$ and $\sQ ( w, \beta )$, respectively. For each $\beta \in Q_+^{\vee}$, let us consider an affine closed subset $\mathbf I ^{\le \beta} \subset \mathbf I$ so that its action on $V ( \lambda ) \otimes \C [z]$ contains matrix entries of degree at most $\left< \beta, \lambda \right>$. We have $\mathbf I^{\le \beta} \cdot \mathbf I ^{\le \beta'} \subset \mathbf I ^{\le \beta + \beta'}$ for each $\beta, \beta' \in Q_+^{\vee}$ and $\mathbf I = \bigcup_{\beta \in Q^{\vee}_+} \mathbf I ^{\le \beta}$. Taking (\ref{Pemb}) into account, we deduce an ind-action
$$\mathbf I ^{\le \gamma} \cdot \sQ ( \beta ) \longrightarrow \sQ ( \beta + \gamma ) \hskip 3mm \text{for each} \hskip 3mm \beta, \gamma \in Q_+^{\vee}$$
that is compatible with $\imath ^{\beta,\beta'}$.

The ind-action of $\mathbf I$ on $\sQ$ preserves $\sQ ( w )$ for each $w \in W$ since $\mathtt{ev} _0 ( \mathbf I ) = B$.

\begin{thm}[Braverman-Finkelberg \cite{BF14a} Theorem 1.2]\label{BFnormal}
For each $\beta \in Q_+^{\vee}$, the variety $\sQ ( \beta )$ is normal. \hfill $\Box$
\end{thm}

By taking the formal expansions of maps along $0$, we have a natural $G[z]$-equivariant embedding $\sQ \hookrightarrow \mathbf Q$ into an infinite type scheme $\mathbf Q$ that contains $G [\![z]\!] / ( H \cdot U [\![z]\!] )$ as its open subset. The scheme $\mathbf Q$ admits a natural $G [\![z]\!]$-action extending that of $G[z]$ (that is realized by replacing $\C[z]_{\le k}$ with $\C [\![z]\!]$ in (\ref{Pemb})). We have a $G[\![z]\!]$-subscheme $\mathbf Q_{0} \subset \mathbf Q$ that has an evaluation at $z = 0$. Hence, we have $\mathbf Q ( w )$ in a parallel fashion to $\sQ$. They admit a natural action of the completed version $\mathbf I ^{\wedge}$ of $\mathbf I$ (we also define the completed version $\mathbf I^{\wedge}_i$ of $\mathbf I_i$ for each $i \in \mathtt I_{\mathrm{aff}}$). By construction, we have $\sQ ( w ) = \sQ \cap \mathbf Q ( w )$ for each $w \in W$.

\begin{lem}\label{Zdense}
The ind-action of $\mathbf I$ on $\sQ ( w )$ has a Zariski dense orbit.
\end{lem}

\begin{proof}
The inclusion $\sQ ( w ) \subset \mathbf Q ( w )$ is dense, and the latter has an open dense orbit with respect to the $\mathbf I^{\wedge}$-action. There exists an $\mathbf I$-indorbit whose closure in $\mathbf Q ( w )$ contains an open dense $\mathbf I^{\wedge}$-orbit as the closure of the orbit of the set of $\C [z]$-valued points of one-parameter subgroups of $G$ is the same as that of $\C [\![z]\!]$-valued points. Such an $\mathbf I$-indorbit must be Zariski dense as required.
\end{proof}

We define
$$H^i ( \sQ, \cO _{\sQ} ( \lambda) ) := \varprojlim_{\beta} H^i ( \sQ ( \beta ), \cO _{\sQ ( \beta )} ( \lambda ) ) \hskip 3mm \text{for every} \hskip 3mm i \in \Z.$$

\begin{thm}[Braverman-Finkelberg \cite{BF14b, BF14c}]\label{BFmain}
For each $\lambda \in \Lambda$, we have a natural isomorphism
$$H^i ( \sQ, \cO _{\sQ} ( \lambda) )^* \cong \begin{cases} W ( \lambda ) & (i=0, \lambda \in \Lambda_+) \\ \{ 0 \} & (\text{otherwise})\end{cases}$$
of graded $\g [z]$-modules $($where the grading arises from the loop rotation$)$. \hfill $\Box$
\end{thm}

\begin{cor}\label{amplecone}
A line bundle $\cO_{\sQ} ( \lambda )$ is very ample if and only if $\left< \alpha ^{\vee}_i, \lambda \right> > 0$ for every $i \in \mathtt I$.
\end{cor}

\begin{proof}
Thanks to (\ref{Pemb}), we know that
$$\sQ \hookrightarrow \prod_{i \in \mathtt I} \P ( V ( \varpi_i ) \otimes \C [z] ).$$
Theorem \ref{free} 1) asserts that $\Gamma ( \sQ, \cO _{\sQ} ( \varpi_i ) )^* \cong W ( \varpi _i ) \rightarrow V ( \varpi_i ) \otimes \C [z]$ is a surjective $\g[z]$-module homomorphism. Therefore, we have
$$\sQ \hookrightarrow \prod_{i \in \mathtt I} \P ( V ( \varpi_i ) \otimes \C [z] ) \dashleftarrow \P ( \otimes _{i \in \mathtt I} W ( \varpi_i ) ),$$
that prolongs to a commutative diagram of the embedings of $\sQ$. We have
$$( \otimes _{i \in \mathtt I} W ( \varpi_i ) ) \otimes_{\C [ z_i; i \in \mathtt I ]} \C ( z_i; i \in \mathtt I ) \cong W ( \rho ) \otimes_{\C [ z_i; i \in \mathtt I ]} \C ( z_i; i \in \mathtt I )$$
by Theorem \ref{free} 5), that implies
$$\sQ \subset \P ( W ( \rho ) ) \subset \P ( W ( \rho ) \otimes_{\C [ z_i; i \in \mathtt I ]} \C ( z_i; i \in \mathtt I ) ).$$
In particular, $\cO_{\sQ} ( \rho ) = \bigotimes_{i \in \mathtt I} \cO_{\sQ} ( \varpi_i )$ is a very ample sheaf of $\sQ$. In general, $\cO_{\sQ} ( \lambda - \rho )$ has a non-zero global section by Theorem \ref{BFmain}, and we have an embedding $\cO_{\sQ} ( \rho )\hookrightarrow \cO_{\sQ} ( \lambda - \rho ) \otimes \cO_{\sQ} ( \rho ) \cong \cO_{\sQ} ( \lambda )$ of (pro-)line bundles on $\sQ$, that yields the if statement.

Only if statement is clear since the restriction of $\cO_{\sQ} ( \lambda )$ to the subspace of constant loops is $\cO_X ( \lambda )$, that is base point free if and only if $\left< \alpha ^{\vee}_i, \lambda \right> > 0$ for every $i \in \mathtt I$.
\end{proof}

\section{Ind-scheme structures on $\sQ ( w )$}

We retain the setting of the previous section.

\begin{defn}[Ind-systems]
Let $w \in W$. An increasing sequence of closed subsets
$$
\mathfrak X_1 \subset \mathfrak X_2 \subset \mathfrak X_3 \subset \cdots \subset \sQ ( w )
$$
of finite type is said to be an ind-system of $\sQ ( w )$ if $\bigcup_{k \ge 1} \mathfrak X_k = \sQ ( w )$ and for every $N \in \Z$, there exists $\beta \in Q_+^{\vee}$ so that $\mathfrak X_N \subset \sQ ( w, \beta )$, and for every $\beta \in Q_+^{\vee}$, there exists $N \in \Z$ so that $\sQ ( w, \beta ) \subset \mathfrak X_N$.
\end{defn}

\begin{lem}\label{indep}
Let $w \in W$ and $\lambda \in \Lambda$. Fix an ind-system $\{ \mathfrak X_k \}_{k \ge 1}$ of $\sQ ( w )$. For each $i \in \Z$, we have
$$\varprojlim_k \, H ^i ( \mathfrak X_k, \cO _{\mathfrak X_k} ( \lambda ) ) = H ^i ( \sQ ( w ), \cO _{\sQ ( w )} ( \lambda ) ).$$
\end{lem}

\begin{proof}
The LHS is the limit through a projective system $H ^i ( \mathfrak X_{k+1}, \cO _{\mathfrak X_{k+1}} ( \lambda ) ) \rightarrow H ^i ( \mathfrak X_k, \cO _{\mathfrak X_k} ( \lambda ) )$ for each $k \ge 1$. By the condition of an ind-system, we find $\beta_1, \beta_2 \in Q_+^{\vee}$ for each $M \gg N \in \Z_{\ge 0}$ so that
\begin{align*}
H ^i ( \sQ ( w, \beta_2 ) , \cO _{\sQ ( w, \beta_2 )} ( \lambda ) ) & \rightarrow H ^i ( \mathfrak X_M, \cO _{\mathfrak X_M} ( \lambda ) ) \\
 & \rightarrow H ^i ( \sQ ( w, \beta_1 ) , \cO _{\sQ ( w, \beta_1 )} ( \lambda ) ) \rightarrow H ^i ( \mathfrak X_N, \cO _{\mathfrak X_N} ( \lambda ) ).
\end{align*}
We also find $M, N \in \Z_{\ge 0}$ with the same maps if we fix $\beta_2 \gg \beta_1$. Therefore, two pro-systems factor through each other, which implies
$$\varprojlim_k \, H ^i ( \mathfrak X_k, \cO _{\mathfrak X_k} ( \lambda ) ) = \varprojlim_{\beta} H ^i ( \sQ ( w, \beta ), \cO _{\sQ ( w, \beta )} ( \lambda ) )$$
as required.
\end{proof}

\begin{thm}\label{pnw}
The ind-scheme $\sQ ( w_0 ) = \sQ$ is projectively normal.
\end{thm}

\begin{proof}
The homogeneous coordinate ring $R ( w_0 )$ of $\sQ ( w_0 )$ is obtained as the graded completion of its $\Gm$-finite part $R ^{\#} ( w_0 ) = \oplus_{\lambda \in \Lambda_+} W ( \lambda )^*$ (cf. \cite{BF14b} Theorem 1.5).

Let us fix a collection of non-zero elements $y = \{ y_i \}_{i \in \mathtt I}$ so that $y_i \in W ( \varpi_i )^*$ for each $i \in \mathtt I$. Consider the ring $R_y$ obtained from $R ( w_0 )$ through the localization of $y$. As we fix $y$, there exists $\beta_0 \in Q_+^{\vee}$ so that the image of $y_i$ in $H^0 ( \sQ ( \beta ), \cO_{\sQ ( \beta )} ( \varpi_i ))$ is non-zero for each $i \in \mathtt I$ when $\beta > \beta_0$ (we remind that each $\sQ ( \beta )$ is integral as being normal). Then, the image of $y$ defines an affine ring $R ( \beta )_y$ obtained by the localization of the homogeneous coordinate ring of $\sQ ( \beta )$. By the definition of the homogeneous coordinate ring, we can form the ring $R ( \beta )_y$ only using $H^0 ( \sQ ( \beta ), \cO_{\sQ ( \beta )} ( \lambda ))$ for $\lambda \gg 0$. By the Serre's vanishing theorem, such a rearrangement guarantees the projective system to be surjective, and consequently $R ( \beta )_y$ is a quotient of $R_y$. In such a circumstance, $R_y$ is integral as each $R_y ( \beta )$ is so. Now we assume $R _y$ is not normal to deduce contradiction (in order to prove that $R_y$ is normal). We have a monic equation $P ( X )$ with coefficients in $R _y$ that has a solution in $\mathrm{Frac} \, R_y$, but not in $R_y$. A solution of $P ( X ) = 0$ is written as $X = \frac{a}{b}$ by $a,b \in R_y$. For $\beta \gg 0$, all the coefficients of the equation $P ( X )$, and $a, b \in R _y$ go to non-zero elements of $R ( \beta )_y$. By Theorem \ref{BFnormal}, we find that $a/b = c ( \beta ) \in R ( \beta )_y$ for $\beta \gg 0$. Taking the inverse limit yields an element in $R _y$ that maps to $\{c ( \beta )\}_{\beta \gg 0}$. Therefore, we conclude that $R_y$ is normal. By the definition of DP-data and the embedding (\ref{Pemb}) (cf. the proof of Corollary \ref{amplecone}), the open sets $\cap _{i \in \mathtt I} \{ y_i \neq 0 \}$ cover the whole $\sQ$, and hence $\sQ$ is normal.

It remains to show that the dual of the multiplication map $W ( \lambda + \mu ) \longrightarrow W ( \lambda ) \otimes W ( \mu )$ is injective for each $\lambda, \mu \in \Lambda_+$ (here we used the fact that the normality of $\sQ$ is equivalent to that of $\P_{\sQ} ( \bigoplus _{i \in \mathtt I} \cO _{\sQ} ( \varpi_i )^{\vee} )$). Here this map extends the (dual) multiplication map $V ( \lambda + \mu ) \hookrightarrow V ( \lambda ) \otimes V ( \mu )$, that is uniquely determined up to scalar as $\g$-modules. Note also that $\C [\A^{(\lambda)}] \otimes \C [\A^{(\mu)}]$ is a free $\C [\A^{(\lambda+\mu)}]$-module of rank $\frac{(\lambda+\mu)!}{\lambda!\mu!}$. Thanks to Theorem \ref{free} 5), a generic specialization along $x \in \A^{(\lambda + \mu)}$ yields an inclusion
$$
\xymatrix{
W ( \lambda + \mu )\otimes _{\C [\A^{(\lambda+\mu)}]} \C_x \ar[r] \ar[d]_{\cong}&  ( W ( \lambda ) \otimes W ( \mu ) ) \otimes _{\C [\A^{(\lambda+\mu)}]} \C_x \ar[d]^{\cong}\\
\bigotimes _{i \in \mathtt I} \bigotimes _{j=1} ^{\left< \alpha_i ^{\vee}, \lambda + \mu \right>} W ( \varpi_i, x_{i,j} ) \ar@{^{(}->}[r] & (  \bigotimes _{i \in \mathtt I} \bigotimes _{j=1} ^{\left< \alpha_i ^{\vee}, \lambda + \mu \right>} W ( \varpi_i, x_{i,j} ) ) ^{\oplus \frac{(\lambda+\mu)!}{\lambda!\mu!}}},$$
where $\{ x_{i,j} \}$ is a set of points in $\C$ determined by the configuration of $x$ (as the map is non-zero and a non-zero $\g [z]$-module endomorphism of $W (\varpi_i,x)$ must be an isomorphism). Since any $\C [\A^{( \lambda+\mu )} ]$-submodule of a free $\C [ \A^{( \lambda+\mu )} ]$-module of finite rank has no torsion element (that is supported on some closed subset of $\A^{( \lambda+\mu )}$), we conclude that the map $W ( \lambda + \mu ) \longrightarrow W ( \lambda ) \otimes W ( \mu )$ must be an inclusion as required.
\end{proof}

\begin{defn}[Demazure modules]
For $\lambda \in \Lambda_+$ and $w \in W$, we have a unique vector $v_{w \lambda} \in V ( \lambda ) \subset W ( \lambda )$ of $\h$-weight $w \lambda$ up to scalar. We define
$$W ( \lambda )_w := U ( \gI ) v_{w \lambda} \subset W ( \lambda )$$
and call it the Demazure submodule of $W ( \lambda )$. By Theorem \ref{qwc}, we also define a $U^{\ge 0}_t$-submodule $W_t ( \lambda )$ generated by a vector with its $U_t^0$-weight $w \lambda$ at degree $0$. We note that $W ( \lambda ) = W ( \lambda ) _{w_0}$ and $W _t ( \lambda ) = W_t ( \lambda ) _{w_0}$.
\end{defn}

\begin{cor}[of the proof of Proposition \ref{pnw}]\label{wmul}
For each $\lambda, \mu \in \Lambda_+$ and $w \in W$, we have an injective map $m_{\lambda, \mu}^w : W ( \lambda + \mu )_w \longrightarrow W ( \lambda )_w \otimes W ( \mu )_w$. \hfill $\Box$
\end{cor}

For each $w \in W$, we define a ring (that generalizes $R ( w_0 )$ in the proof of Proposition \ref{pnw})
$$R ^{\#} ( w ) :=  \bigoplus _{\lambda \in \Lambda_+} W ( \lambda ) _w ^*,$$
where the product structure is given by Corollary \ref{wmul}. Let $R ( w )$ denote the $\Gm$-graded completion of $R ( w )$, taken $\Lambda_+$-degreewise.

\begin{cor}\label{pnwc}
The ring $R  ( w_0 )$ is normal. \hfill $\Box$
\end{cor}

\begin{cor}[of the proof of Proposition \ref{pnw}]\label{homog2}
The ind-scheme $\sQ ( w )$ is projectively normal if the ring $R ( w )$ is normal and $R ^{\#} ( w )$ defines a dense subring of the projective coordinate ring of $\sQ ( w )$.
\end{cor}

\begin{proof}
Our ind-system (used in the definition of $\sQ ( w )$ through $\sQ$) is equivalent to these obtained by cutting out by the degrees by its definition (cf. (\ref{Pemb})). Therefore, if $R^{\#} ( w )$ is a dense subring of the projective coordinate ring of $\sQ ( w )$, then the latter is $R( w )$.
\end{proof}

\section{Main Results}

We continue to work in the setting of the previous section.

\begin{defn}[Demazure operator]\label{do}
For each $i \in \mathtt I_{\mathrm{aff}}$, we define a linear operator on $\C (\!(q)\!) [\Lambda]$ by
$$D_i ( q^m e ^{\lambda} ) := q^m \frac{e^{\lambda} - e^{s_i \lambda - \alpha_i}}{1 - e^{- \alpha_i}} \hskip 3mm \text{for each} \hskip 3mm m \in \Z \text{ and } \lambda \in \Lambda,$$
where we formally put $q = e ^{\delta}$. For $w \in W_{\mathrm{aff}}$, we fix a reduced expression $s_{i_1} s_{i_2} \cdots s_{i_{\ell}}$ of $w$ and set
$$D_w := D_{i_1} \circ D_{i_2} \circ \cdots \circ D_{i_{\ell}}.$$
\end{defn}

\begin{thm}[Demazure-Joseph, cf. Kumar \cite{Kum02} \S V\!I\!I\!I]\label{DCF}
We have:
\begin{enumerate}
\item For each $w \in W_{\mathrm{aff}}$, the Demazure operator $D_w$ is independent of the choice of a reduced expression;
\item For each $\lambda \in \Lambda$ and $w \in W$, we have
$$\sum_{i \ge 0} (-1)^i \mathrm{ch} \, H ^i ( X ( w ), \cO _{X ( w )} ( \lambda ) )^* = D_w ( e^{\lambda} );$$
\item For each $\lambda \in \Lambda_+$ and $w \in W$, we have $H^0 ( X ( w ), \cO _{X ( w )} ( \lambda ) )^* \cong U ( \gb ) v_{w \lambda}$ as $B$-modules and $H ^i ( X ( w ), \cO _{X ( w )} ( \lambda ) ) = \{ 0 \}$ for $i > 0$;
\item For each $w \in W$, the restriction through $X ( w ) \subset X$ induces a $B$-module inclusion $H^0 ( X ( w ), \cO _{X ( w )} ( \lambda ) )^* \subset V ( \lambda )$.\hfill $\Box$ 
\end{enumerate}
\end{thm}

\begin{lem}\label{rest}
For each $\lambda \in \Lambda_+$ and $w \in W$, the space $\Gamma ( \sQ ( w ), \cO _{\sQ (w)} ( \lambda) ) ^*$ contains a non-zero vector of weight $w \lambda$ arising from $\Gamma ( X ( w ), \cO _{X (w)} ( \lambda ) )^*$.
\end{lem}

\begin{proof}
We have $0 \neq v_{w\lambda} \in \Gamma ( X ( w ), \cO _{X (w)} ( \lambda ) )^* \subset \Gamma ( X, \cO _X ( \lambda ))^*$ by Theorem \ref{DCF} 3) and 4). We have an inclusion $X ( w ) \subset \sQ ( w )$ of constant quasimaps with their defects supported at $\infty$, that presents a section of $\mathtt{ev}_0$. The degree $0$-part of the map $\sQ \rightarrow \P \Gamma ( \sQ, \cO_{\sQ} ( \lambda) ) ^*$ represents the image of the evaluation map $\sQ_0 \rightarrow X$. In particular, we have $[v_{w \lambda}] \in X ( w ) \subset \sQ ( w )$. Being a unique vector of weight $w\lambda$ at degree $0$ in $W ( \lambda )$, the dual vector $v^*$ of $v_{w \lambda}$ in $\Gamma ( \sQ, \cO _{\sQ} ( \lambda) )$ is uniquely determined up to scalar. Since $v^*$ defines a non-zero regular function on $\sQ ( w )$, it survives through the restriction to $\Gamma ( \sQ (w), \cO _{\sQ ( w )} ( \lambda) )$. Hence, we deduce $v_{w \lambda} \in ( \mathrm{Im} \, \Gamma ( \sQ, \cO _{\sQ} ( \lambda) ) \to \Gamma ( \sQ ( w ), \cO _{\sQ (w)} ( \lambda) ) )^*$. Therefore, $v_{w \lambda}$ must prolong to $\Gamma ( \sQ (w), \cO _{\sQ ( w )} ( \lambda) ) ^*$ as required.
\end{proof}

\begin{lem}\label{d-est}
Let $V$ be a graded $\g$-module with finitely many distinct $\h$-weights. Let $E \subset V$ be its $\gb$-submodule. For each $i \in \mathtt I$, we have
$$U ( \gp_i ) E \subset H ^0 ( \P^1, P_i \times^{B} E )^* \hskip 2mm \text{and} \hskip 2mm \mathrm{ch} \, U ( \gp_i ) E \le D_i ( \mathrm{ch} \, E ),$$
where the latter equality holds if and only if $V$ has a finite $\gp_i$-filtration
$$\{ 0 \} \subset V (N) \subset V (N-1) \subset V ( N-2 ) \subset \cdots \subset V ( 0 ) = V,$$
so that the induced associated graded 
$$\bigoplus_{k = 0}^{N-1} ( E \cap V ( k ) ) / ( E \cap V ( k + 1 ) )$$
is a direct sum of irreducible $L_i ^0$-modules and one-dimensional $\gb$-modules $\C_\gamma$ so that $\left< \alpha_i^{\vee}, \gamma \right> > 0$. The analogous assertion also holds for a $U_t^{\flat}$-module $V$ and its $(U_t^{\flat} \cap U_t^{\ge 0})$-submodule $E$.
\end{lem}

\begin{proof}
During the proof of Lemma \ref{d-est}, we set $\mathop{SL} ( 2 ) := L_i ^0 \subset P _i$ and $\mathfrak{sl}_2 := \mathrm{Lie} \, L_i ^0$. Since $E$ is assumed to be $\gb$-stable, we have $U ( \gp_i ) E = U ( \mathfrak{sl}_2 ) E$. Hence, we replace $\gp_i$ and $P_i$ with $\mathfrak{sl}_2$ and $\mathop{SL} ( 2 )$ during this proof. We also use $B$ to represent $B^0_i$ $(= L_i^0 \cap B)$ for bravity. Moreover, we identify $\Lambda_+$ with $\Z_{\ge 0}\varpi$, where $\varpi$ is the fundamental weight of $\mathfrak{sl} ( 2 )$ so that $\cO _{\mathbf I_i/\mathbf I} ( \varpi_i ) = \cO _{\mathop{SL} ( 2 )/B} ( \varpi ) \cong \cO_{\P^1} ( 1 )$.

We have a natural inclusion $U ( \mathfrak{sl}_2 ) E \subset H ^0 ( \P^1, \mathop{SL} ( 2 ) \times^B E )^*$ coming from the  restriction to the $\mathfrak{sl}_2$-highest weight part of $E$ regarded as a fiber at $B/B$. The inequality is easy to verify when $V$ is irreducible, and we deduce the inequality part of the assertion by the Euler-Poincar\'e principle in general (as the LHS is subadditive and the RHS is additive with respect to a short exact sequence). 

In case $E$ admits such a filtration, each graded piece define subquotients of $\mathop{SL} ( 2 ) \times^B E \subset \mathop{SL} ( 2 ) \times ^B V$ whose direct summands are of the form $V ( \lambda ) \otimes \cO_{\P^1} = V ( \lambda ) \otimes \cO _{\P^1}$ or $V ( \lambda ) \otimes \cO _{\P^1} \rightarrow \!\!\!\!\! \rightarrow \cO _{\P^1} ( \lambda )$ for some $\lambda \in \Z_{\ge 0} \varpi = \Lambda_+$ (the former case corresponds to irreducible $L_i ^0$-modules and the latter case corresponds to one-dimensional $\gb$-modules). In all cases, we have $H ^1 ( \P^1, \bullet ) = \{ 0 \}$, and a successive applications of short exact sequences yields if part of the assertion on $\gp_i$-filtrations.

We prove the only if part of the assertion on $\gp_i$-filtrations. For each $k \ge 0$, we define $V [k]$ to be the $\mathfrak{sl}_2$-direct summand of $V$ whose highest weight is $k \varpi$ (via the restriction from $\gp_i$). Consider the filtration
$$\{ 0 \} \subset V (N) \subset V (N-1) \subset V ( N-2 ) \subset \cdots \subset V ( 0 ) = V,$$
where $V ( k ) = V ( k+1 ) \oplus V [k]$ for each $k \ge 0$ (and we have $V ( k ) = \{ 0 \}$ for $k \gg 0$). Note that each $V [k]$ and $V ( k )$ inherit the grading and the $\h$-module structure from $V$. We define $E ( k ) := E \cap V ( k )$ for each $k \ge 0$. Each $E ( k )$ is stable by the $\gb$-action. We assume $N'$ to be the largest number so that $E ( N' ) / E ( N'+1 )$ is not a direct sum of $\mathfrak{sl}_2$-modules and one-dimensional $\gb$-modules of weight $\Z_{>0} \varpi$ to deduce contradiction. We have 
$$\mathrm{ch} \, H ^0 ( \P^1, \mathop{SL} ( 2 ) \times^B E (N') ) ^* - \mathrm{ch} \, U ( \mathfrak{sl}_2 ) E ( N' ) > \mathrm{ch} \, H ^1 ( \P^1, \mathop{SL} ( 2 ) \times^B E (N') ) ^*$$
from the $\gb$-invariance of the $E(N')$ and the hypothesis (with the help of Euler-Poincar\'e principle). This is the same as an inequality
\begin{equation}
\mathrm{ch} \, U ( \mathfrak{sl}_2 ) E ( N' ) < D_i ( \mathrm{ch} \, E (N') ). \label{ineqA}
\end{equation}

For each $v \in E \backslash E ( N' )$ so that $v' \in U ( \gb ) v$ has $\mathfrak{sl}_2$-weight $k \varpi$ for $k \ge N'$, we have $v' \in E ( N' )$ by a weight counting. In particular, we have
$$\left( v + U ( \mathfrak{sl}_2 ) E ( N' ) \right) \cap \bigoplus_{k < N'} V [ k ] \neq \emptyset.$$
This forces
$$\left( U ( \mathfrak{sl}_2) E \right) / \left( U ( \mathfrak{sl}_2) E (N') \right) \cong U ( \mathfrak{sl}_2) \left( E / E ( N' ) \right) \subset \bigoplus_{k < N'} V [ k ].$$
In particular, we have an inequality
\begin{equation}
\mathrm{ch} \, \left( U ( \mathfrak{sl}_2) E \right) / \left( U ( \mathfrak{sl}_2) E (N') \right) = \mathrm{ch} \, U ( \mathfrak{sl}_2) \left( E / E ( N' ) \right) \le D_i ( \mathrm{ch} \, E / E ( N' ) )\label{ineqB}
\end{equation}
The inequalities (\ref{ineqA}) and (\ref{ineqB}) result in
$$\mathrm{ch} \, \left( U ( \mathfrak{sl}_2) E \right) < D_i ( \mathrm{ch} \, E ( N' ) ) + D_i ( \mathrm{ch} \, E / E ( N' ) ) = D_i ( \mathrm{ch} \, E ).$$
Therefore, we have no possible choice of $N'$. Hence the only if part of the assertion on $\gp_i$-filtrations follows.

Since the integrable representation theory of $U _t ( \mathfrak{sl}_2 )$ (with $t$ being generic) and $U ( \mathfrak{sl}_2 )$ are the same, exactly the same proof works in the quantum setting as required.
\end{proof}

\begin{defn}
For $w \in W$ and $i \in \mathtt I_{\mathrm{aff}}$, we define $\overline{s_i w} > _q w$ if we have $s_i w > w$ (when $i \in \mathtt I$) or $w^{-1} \vartheta \not\in \Delta^+$ (when $i = 0$).
\end{defn}

\begin{thm}[LNSSS-I \cite{LNSSS1} \S 6]\label{ind-w}
For every $w, v \in W$, there exists a sequence $i_1,i_2,\ldots,i_{\ell} \in \mathtt I_{\mathrm{aff}}$ so that
\begin{equation}
w = \overline{s_{i_1} s_{i_2} \cdots s_{i_{\ell}} v} >_q \overline{s_{i_2} \cdots s_{i_{\ell}} v} >_q \cdots >_q \overline{s_{i_{\ell}} v} >_q v.\label{q-order}
\end{equation}
\end{thm}

\begin{proof}
The relation $>_q$ without taking the projection $W_{\mathrm{aff}} \to W$ generates an order in $W_{\mathrm{aff}}$. It is a variant of the quantum (or generic) Bruhat order in the sense that the weak Bruhat order is different from the Bruhat order (cf. Ishii-Naito-Sagaki \cite{INS16} Appendix A.3 and Bjorner-Brenti \cite{BB05}). Therefore, the assertion is included in Lenart-Naito-Sagaki-Schilling-Shimozono \cite{LNSSS1} \S 6.
\end{proof}

\begin{thm}[Kashiwara \cite{Kas05} 2.8, Naito-Sagaki \cite{NS14} \S 5]\label{NS}
Let $\lambda \in \Lambda _+$ and let $w \in W$. For each $i \in \mathtt I_{\mathrm{aff}}$ such that $\overline{s_i w} >_{q} w$, we have an identity
$$D_i ( \mathrm{ch} \, W_t ( \lambda )_{w} ) = q ^{\delta_{i,0} \left< \vartheta^{\vee}, w \lambda \right>} \cdot \mathrm{ch} \, W_t ( \lambda )_{\overline{s_i w}}.$$
\end{thm}

\begin{proof}
In view of Lemma \ref{d-est}, the assertion follows if the $\mathfrak{sl}_2$-crystal (corresponding to $i \in \mathtt I$) structure of $W _t ( \lambda )_{w}$ inside $W _t ( \lambda )_{s_i w}$ is a disjoint union of genuine $\mathfrak{sl}_2$-crystals and Demazure crystals (it is a crystal with one element with weight $\gamma$ so that $\left< \alpha_i ^{\vee}, \gamma \right> > 0$ in this case).

The assertion on crystal itself follows by \cite{Kas05} Lemma 2.7 as the crystal basis there is equal to these of $W_t ( \lambda )_{s_i w}$ as $\mathfrak{sl}_2$-crystals (cf. \cite{Kas05} \S 2.5, see also \cite{NS14} proof of Proposition 5.1.1).
\end{proof}

\begin{cor}\label{cNS}
Let $\lambda \in \Lambda _+$ and let $w \in W$. For each $i \in \mathtt I_{\mathrm{aff}}$ such that $\overline{s_i w} >_q w$, we have an identity
$$D_i ( \mathrm{ch} \, W ( \lambda )_{w} ) = q^{\delta_{i,0} \left< \vartheta^{\vee}, w \lambda \right>}  \cdot \mathrm{ch} \, W ( \lambda )_{\overline{s_i w}}.$$
\end{cor}

\begin{proof}
For each $w \in W$, we set $W' ( \lambda )_{w}$ to be the specialization of a module $W_t ( \lambda )_{w}$ by setting $t = 1$ in their $\C [t]$-lattice spanned by the global bases. By using a $\C [t]$-lattice of $W_t ( \lambda )_w \subset W_t ( \lambda )$, the specialization map $t \to 1$ yields an $\gI$-module inclusion $W' ( \lambda )_w \subset W ( \lambda )$. Since $W' ( \lambda )_w$ shares a vector $v_{w \lambda}$ with $W ( \lambda )_w$, we have $W ( \lambda )_w \subset W' ( \lambda )_w$. In particular, we have $\mathrm{ch} \, W ( \lambda )_w \le \mathrm{ch} \, W' ( \lambda )_w$ for each $w \in W$. By Theorem \ref{qwc}, this is an equality for $w = w_0$.

We prove the assertion on induction on $w \in W$ from $w_0$ using $\{s_i\}_{i \in \mathtt I_{\mathrm{aff}}}$, that is possible in view of Theorem \ref{ind-w} and the fact that Demazure submodules of an extremal weight module are parametrized by $w \in W_{\mathrm{aff}}$, and the right multiplication of $t_{\beta}$ induces an automorphism of an extremal weight module that yields an isomorphism $W' ( \lambda )_{\overline{w}} \cong W' ( \lambda )_{\overline{w t_{\beta}}}$ (see \cite{Kas05} (2.26)). Let $i \in \mathtt I_{\mathrm{aff}}$ so that $\overline{s_i w} >_q w$. Since we have $W ( \lambda )_w = W' ( \lambda )_w$, we have
\begin{equation}
U ( \gI_i ) W ( \lambda )_w = U ( \gI_i ) W' ( \lambda )_w \subset W' ( \lambda )_{\overline{s_i w}}.\label{t=0incl}
\end{equation}
As in the proof of Theorem \ref{NS}, the global basis is compatible with the embedding $W _t ( \lambda )_{w} \subset W_t ( \lambda )_{\overline{s_iw}}$. Every global basis element of $W_t ( \lambda )_{\overline{s_iw}}$ labeled by a highest weight element viewed as a $\mathfrak{sl}_2$-crystal (corresponding $i \in \mathtt I$) belongs to $W _t ( \lambda )_{w}$. In view of Lemma \ref{d-est} and Theorem \ref{NS}, a $U_t ( \mathfrak{sl}_2 )$-highest weight vector of $W_t ( \lambda )_{\overline{s_i w}}$ is contained in $W_t ( \lambda )_w$ with grading shift $\left< \vartheta^{\vee}, w \lambda \right>$ when $i=0$. By the comparison of characters, we deduce that the dimension of the space of $U_t ( \mathfrak{sl}_2 )$-highest weight vectors of $W_t ( \lambda )_{\overline{s_i w}}$ with given weight and degree coincides with the number of highest weight elements of the Demazure crystal of $W _t ( \lambda )_{\overline{s_i w}}$ with the same weight and degree (that is finite). By the multiplication rule of the global bases (see e.g. \cite{Kas05} Definition 2.4 iii)), we deduce that a sum of global basis elements (of a fixed weight) corresponding to non-highest weight elements viewed as $\mathfrak{sl}_2$-crystal never gives rise to a non-zero $U ( \mathfrak{sl}_2 )$-highest weight vector by reduction mod $(t-1)$. Therefore, we cannot have a $\mathfrak{sl}_2$-highest weight vector in $W' ( \lambda )_{\overline{s_iw}} \backslash W' ( \lambda ) _w$ with a given $\h$-weight and degree. It follows that $U ( \mathfrak{sl}_2 ) W' ( \lambda ) _w = W' ( \lambda )_{\overline{s_iw}}$. Thus, the inclusion in (\ref{t=0incl}) is in fact an equality.

By the PBW theorem, we have $W ( \lambda )_{\overline{s_i w}} \cong U ( \mathfrak{sl}_2 ) W ( \lambda )_w$. Now Theorem \ref{NS} implies
$$D_i ( \mathrm{ch} \, W ( \lambda )_{w} ) = D_i ( \mathrm{ch} \, W' ( \lambda )_{w} ) = q^{\delta_{i,0} \left< \vartheta^{\vee}, w \lambda \right>} \cdot \mathrm{ch} \, W' ( \lambda )_{\overline{s_iw}}  = q^{\delta_{i,0} \left< \vartheta^{\vee}, w \lambda \right>} \cdot \mathrm{ch} \, W ( \lambda )_{\overline{s_iw}},$$
which proceeds the induction as required.
\end{proof}

\begin{prop}\label{inf-norm}
For each $w \in W$, the ring $R ( w )$ is normal.
\end{prop}

\begin{proof}
For each $\beta \in Q_+^{\vee}$ and $i \in \mathtt I$, we have a $\g [z]$-module embedding $V ( \varpi_i ) \otimes \C [z] \rightarrow V ( \varpi_i ) \otimes \C [z]$ induced by the multiplication by $z ^{\left< \beta, \varpi_i \right>}$. Taking their projective covers extends this map to an embedding $W ( \varpi_i ) \hookrightarrow W ( \varpi_i )$ by Theorem \ref{free}. By Corollary \ref{wmul}, this induces a $\g [z]$-module embedding $W ( \lambda ) \hookrightarrow W ( \lambda )$ induced by the multiplication of the product of $\left< \beta, \varpi_i \right>$-th power of a degree $\lambda_i$ primitive generator of $\C [\A^{(\lambda_i)}] \subset \C [\A^{(\lambda)}]$ in Theorem \ref{free} 2).

This endomorphism is the same (up to scalar) as the action of a lift of $t_{\beta} \in W_{\mathrm{aff}}$ to $H (z)$ in view of the embedding (\ref{Pemb}) (with an extension of the scalar to $\C (z)$ if necessary). In addition, it also corresponds to the twist of cyclic vectors of Demazure modules corresponding to $D_{t_{\beta}}$ in accordance with Corollary \ref{cNS}. Therefore, it extends to an inclusion $W ( \lambda )_w \hookrightarrow W ( \lambda )_w$ for each $w \in W$. It further gives rise to a surjection $R ( w ) \longrightarrow \!\!\!\!\! \rightarrow R ( w )$ of algebras induced by each $\beta \in Q^{\vee}_+$. Hence, the definition of $R ( w )$ can be naturally extended to $w \in W _{\mathrm{aff}}$, with the difference by a translation part gives rise to an isomorphic algebra with degree twists in accordance with Corollary \ref{cNS}. (These are rephrasements of the inclusions $\sQ \hookrightarrow \sQ$ and $\sQ ( w ) \hookrightarrow \sQ ( w )$ given by twisting defects supported on $0$, though the latter is yet to be established.) In view of this, we can prove the assertion by induction on $>_q$ using Theorem \ref{ind-w}. The case $w = w_0$ is Corollary \ref{pnwc}. We assume the assertion for $w \in W$ and find $i \in \mathtt I_{\rm{aff}}$ so that $\overline{s_i w} >_q w$.

The algebra $R ( w )$ admits a $B^0_i$-module structure. In addition, we can write $R ( w ) := \varprojlim_{m} \, R ( w )_m$, where $\{ R ( w )_m \}_m$ is a suitable surjective projective system of $( H \cdot B^0_i )$-stable graded quotients of $R ( w )$ that are ($\Lambda_+$-graded componentwise) finite dimensional vector spaces of bounded degrees (thanks to the fact that each graded component of $W ( \lambda )_w$ is finite-dimensional, we can deduce that all projective systems yield the same topological ring, and hence the choice of $R( w ) _m$ is not important). We form an ind-vector bundle $\mathcal R_i ( w ) := \varinjlim_m \, \mathop{SL} ( 2 ) \times^{B^0_i} R ( w )_m ^*$ over $\P^1$. Fix $x \in \P^1$, and find a local coordinate $t_x$ of $x$. We have $\C [t_x]_{(0)} \cong \cO_{\P^1,x}$ as a ring, where $(0)$ denote the localization along $t_x = 0$. The stalk of $\mathcal R_i ( w )$ at $x$ is isomorphic to the scalar extension $R ( w ) \otimes_{\C} \C [t_x]_{(0)}$, and hence is normal.
Now we have
$$H ^0 ( \P^1, \mathcal R_i ( w ) ) = \bigcap _{x \in \P^1} R ( w ) \otimes_{\C} \C [t_x]_{(0)} \subset \mathrm{Frac} ( R ( w ) \otimes_{\C} \C ( \P^1 ) ).$$
Since the intersection of normal rings that shares the same fraction field is normal (by the definition of integral closure), we conclude that the ring $H ^0 ( \P^1, \mathcal R_i ( w ) )$ is normal. By construction, we have $W ( \lambda )_{\overline{s_i w}} = U ( \gp_{i} ) W ( \lambda )_{w} \subset W ( \lambda )$ for each $\lambda \in \Lambda_+$ (with a possible degree twist of $W ( \lambda )$). By Lemma \ref{d-est}, we deduce
\begin{equation}
R^{\#}_{\clubsuit} ( \overline{s_i w} ) = U ( \gp_i  ) R^{\#} ( w ) \hookrightarrow H ^0 ( \P^1, \mathcal R_i ( w ) ),\label{char}
\end{equation}
where $R^{\#}_{\clubsuit} ( \overline{s_iw} )$ is obtained by a degree twist of $W ( \lambda )_{\overline{s_i w}}$ by $\left< \vartheta^{\vee}, w \lambda \right>$ when $i = 0$. In particular, we have an inclusion $R^{\#} ( \overline{s_i w} ) \hookrightarrow H ^0 ( \P^1, \mathcal R_i ( w ) )$ of algebras. Therefore, the comparison of Corollary \ref{cNS} with $(\ref{char})$ forces $R ( \overline{s_i w} ) \cong H ^0 ( \P^1, \mathcal R_i ( w ) )$ (through Lemma \ref{d-est}). This shows that $R ( \overline{s_i w} )$ is a normal ring, and the induction proceeds.
\end{proof}

\begin{lem}\label{1step}
Let $\beta \in Q_+^{\vee}$, $w \in W$, and $i \in \mathtt I$ so that $s_i w > w$. We have a surjective map $q_i : P_i \times ^B \sQ ( w, \beta ) \rightarrow \sQ ( s_i w, \beta )$. Similarly, we have a surjective map $P_i \times ^B \sQ ( w ) \rightarrow \sQ ( s_i w )$ that we denote by the same letter.
\end{lem}

\begin{proof}
The variety $\sQ ( \beta )$ is irreducible, and so is its open subset $\sQ_0 ( \beta )$. Since $X (w)$ is connected, we deduce that $\sQ_0 ( w, \beta )$, and hence $\sQ ( w, \beta )$ is irreducible. As $\sQ ( \beta )$ is projective, so is $\sQ ( w, \beta )$. Therefore, the image of $q_i$ is irreducible and projective. In addition, we have $\sQ _0 ( s_i w, \beta ) \subset \sQ_0 ( \beta ) \cap \mathrm{Im} \, q_i$, that is actually an open dense subset of $\mathrm{Im} \, q_i$. Therefore, we conclude $\sQ ( s_i w, \beta ) = \mathrm{Im} \, q_i$, that implies the first assertion. The second assertion is now clear.
\end{proof}

\begin{lem}\label{1step0}
Let $\beta \in Q_+^{\vee}$ and $w \in W$ so that $w^{-1} \vartheta \not\in \Delta^+$. We have a map $q_0 : \mathop{SL} ( 2 ) \times ^{B _0^0} \overline{B_0^0 \sQ ( w, \beta )} \rightarrow \sQ ( s_{\vartheta} w, \beta + \gamma )$ for some $ \gamma \le 2 \vartheta^{\vee}$ that is independent of $\beta$. Similarly, we have a map $\mathbf I_0 \times ^{\mathbf I} \sQ ( w ) \rightarrow \sQ$ $($that we denote by the same letter$)$ whose image is $\sQ ( s_{\vartheta} w )$ with an appropriate twist of the defect at $0$.
\end{lem}

\begin{proof}
We have a map $\mathop{SL} ( 2 ) \to G ( z )$ so that its image is $L_0^0$. We have a map
$$\mathop{SL} ( 2 ) \times \prod_{i \in \mathtt I} \P ( V ( \varpi_i ) \otimes \C [z] _{\le m} ) \longrightarrow \prod_{i \in \mathtt I} \P ( V ( \varpi_i ) \otimes z^{-m_i} \C [z] _{\le m + 2 m_i} ),$$
where $m_i := \left< \vartheta^{\vee}, \varpi_i \right>$ for $i \in \mathtt I$. This map does not preserve $\sQ ( \beta )$ (in usual and ind- senses), but we see that $B_0^0 [v_{w \varpi_i}] = [v_{w \varpi_i}]$ and $[v_{s_{\vartheta} w \varpi_i} \otimes z^{\left< \vartheta ^{\vee}, w \varpi_i \right>}] \in \mathop{SL} ( 2 ) [v_{w \varpi_i}]$.

By Lemma \ref{Zdense}, the $\mathop{SL} ( 2 )$-multiplication of $\sQ ( w )$ defines a dense subset of the $\mathop{SL} ( 2 )$-multiplication in $\mathbf Q$. Here the set of $\mathbf I^{\wedge}$-orbits of $\mathbf Q$ is parametrized by a subset of $W_{\mathrm{aff}}$ (see \cite{FM} 4.2), and the $\mathop{SL} ( 2 )$-multiplication of the orbit corresponding to $w$ splits into two orbits corresponding to $w$ and $s_0 w$ by the Bruhat decomposition of $\mathop{SL} ( 2 )$. By the above calculation (and the fact that our $\mathop{SL} ( 2 )$-action does not change the defect outside of $0$), we deduce that $\mathop{SL} ( 2 ) \sQ ( w, \beta ) \subset \sQ ( s_{\vartheta} w )$ if we twist the degree of the $i$-th component of the embedding by $\left< w ^{-1} \vartheta ^{\vee}, \varpi_i \right>$. Thus, adjusting the defect (at $0$) so that $( [v_{s_{\vartheta} w \varpi_i} \otimes z^{\left< \vartheta ^{\vee}, w \varpi_i \right>}] )_{i \in \mathtt I} \in \sQ ( s_{\vartheta} w )$ and taking the limit $\beta \to \infty$ yields $\mathrm{Im} \, q_0 = \sQ ( s_{\vartheta} w )$ by Lemma \ref{Zdense}. This proves the second assertion.

Since $\vartheta^{\vee}$ is the highest short coroot, we have $\vartheta ^{\vee} \ge w ^{-1} \vartheta^{\vee} \ge - \vartheta^{\vee}$. For the first assertion, we further need to add $\vartheta^{\vee}$ to take $\gamma = \vartheta^{\vee} - \sum_{i \in \mathtt I} \left< \vartheta, w \varpi_i\right> \alpha_i ^{\vee}$ in order that the adjusted DP-data lands safely in $\sQ ( s_{\vartheta} w, \beta + \gamma )$.
\end{proof}

\begin{thm}\label{main}
For each $\lambda \in \Lambda$ and $w \in W$, it holds:
\begin{enumerate}
\item we have the following isomorphisms as $\gI$-modules:
$$H^i ( \sQ ( w ), \mathcal O _{\sQ ( w )} ( \lambda ) ) ^* \cong \begin{cases} W ( \lambda )_w & (i=0, \lambda \in \Lambda_+)\\ \{0\} & (\text{otherwise}) \end{cases};$$
\item the restriction map $\Gamma ( \sQ, \mathcal O _{\sQ} ( \lambda ) ) \longrightarrow \Gamma ( \sQ ( w ), \mathcal O _{\sQ ( w )} ( \lambda ) )$ is surjective;
\item the indscheme $\sQ ( w )$ is normal and projectively normal.
\end{enumerate}
\end{thm}

\begin{proof}
We first consider the case $w = w_0$. Then, the first assertion follows by Theorem \ref{BFmain}. The second assertion is trivial, and the third assertion follows by Proposition \ref{pnw}.

Since adding defects at $0 \in \P^1$ gives an isomorphic pair of ind-schemes, we prove the assertion by induction on $>_q$ using Theorem \ref{ind-w}. We assume that the assertions hold for $w \in W$ and fix $i \in \mathtt I_{\mathrm{aff}}$ so that $\overline{s_i w} >_q w$. For the sake of simplicity, we denote $\overline{s_iw}$ by $s_i w$ during this proof.

We set $\sQ^+ ( w, \beta ) := \overline{B_i ^0 \sQ ( w, \beta )}$ for each $\beta \in Q_+^{\vee}$. We have $\sQ^+ ( w, \beta ) = \sQ ( w, \beta )$ whenever $i \in \mathtt I$, and $\sQ^+ ( w, \beta )$ forms an ind-structure of $\sQ ( w )$ by Lemma \ref{1step0}. Let us denote the image of $q_0$ in Lemma \ref{1step0} by $\sQ^+ ( s_{\vartheta} w, \beta )$ when $i = 0$. It defines an ind-structure of $\sQ ( s_{\vartheta} w )$ since we have $\sQ ( s_{\vartheta} w, \beta - 2 \vartheta ^{\vee} ) \subset \sQ^+ ( s_{\vartheta} w, \beta )$ for $\beta \gg 0$ (by examining the proof of Lemma \ref{1step0}).

We have an $\gI$-module map
$$\eta : H ^0 ( \sQ ( s_i w ), \cO_{\sQ ( s_i w )} ( \lambda ) ) ^* \rightarrow H ^0( \sQ ( w_0 ), \cO_{\sQ ( w_0 )} ( \lambda ) ) ^* = W ( \lambda )$$
arising from the dual of the restriction map. By Lemma \ref{rest}, we have $W ( \lambda ) _{s_i w} \subset \mathrm{Im} \, \eta$. In particular, we have
\begin{equation}
\mathrm{ch} \, W ( \lambda )_{s_i w} \le \mathrm{ch} \, H ^0 ( \sQ ( s_i w ), \cO_{\sQ ( s_i w )} ( \lambda ) )^*\label{fund}
\end{equation}

By Corollary \ref{cNS} and Lemma \ref{d-est}, the first assertion is equivalent to an isomorphism
\begin{equation}
H ^k ( \sQ ( s_i w ), \cO_{\sQ ( s_i w )} ( \lambda ) ) \cong H ^{k} ( \P^1, \mathop{SL} ( 2 ) \times ^{B_i^0} W ( \lambda )_w ) \hskip 3mm \text{for each} \hskip 3mm k \in \Z.\label{1st}
\end{equation}

By assumption and Lemma \ref{indep}, we deduce that
\begin{equation}
H ^k ( \sQ ( w ), \cO _{\sQ ( w )} ( \lambda ) ) \cong \varprojlim_{\beta} H ^k ( \sQ^+ ( w, \beta ), \cO _{\sQ^+ ( w, \beta )} ( \lambda ) ) \hskip 3mm \text{for each} \hskip 3mm k \in \Z.\label{ch-ind}
\end{equation}
We set $\sQ^+ ( i, w, \beta ) := \mathop{SL} ( 2, \C ) \times ^{B^0_i} \sQ^+ ( w, \beta )$. We have a commutative diagram:
\begin{equation}
\xymatrix{
\sQ^+ ( i, w, \beta ) \ar[r]^-{q_i} \ar[d]_{h_i} & \sQ ^+ ( s_i w, \beta ) \ar[d]^{\bar{h}_i}\\
\P ^1 \ar[r]^{\bar{q}_i} & \mathrm{pt}
}.\label{Qcomm}
\end{equation}

\begin{claim}\label{change}
We have $($the limit of convergent$)$ spectral sequences
$$\varprojlim_{\beta} \R ^u ( \overline{h}_i )_* ( \R ^r (q_i)_* \cO_{\sQ^+ ( i, w, \beta )} ( \lambda ) ) \Rightarrow \R ^{u+r} ( \overline{q}_i )_* ( \varprojlim_{\beta} \, (h_i)_* \cO_{\sQ^+ ( i, w, \beta )} ( \lambda ) ).$$
\end{claim}

\begin{proof}
By (\ref{ch-ind}) and the induction hypothesis, the fiber $H ^{u} ( \sQ ( w, \beta ), \cO_{\sQ ( w, \beta )} ( \lambda ) )$ of the pro-sheaf $\varprojlim _{\beta}\R ^u (h_i)_* \cO_{\sQ^+ ( i, w, \beta )} ( \lambda )$ satisfies the Mittag-Leffler condition for each fixed degree. In addition, the effect of $( \bar{q}_i )_*$ changes degrees at most by $2 \left< \vartheta^{\vee}, \lambda \right>$. Therefore, $\varprojlim$ commutes with $\mathbb R^r ( \bar{q}_i )_*$, and the (limit of the) Leray spectral sequence
$$\varprojlim_{\beta} \R ^r ( \bar{q}_i )_* \left( \R ^u (h_i)_* \cO_{\sQ^+ ( i, w, \beta )} ( \lambda ) \right) \Rightarrow \varprojlim_{\beta} H ^{u+r} ( \sQ^+ ( i, w, \beta ), \cO_{\sQ^+ ( i, w, \beta )} ( \lambda ) )$$
gives rise to the spectral sequence
$$\R ^r ( \bar{q}_i )_* \left( \varprojlim_{\beta} \R ^u (h_i)_* \cO_{\sQ^+ ( i, w, \beta )} ( \lambda ) \right) \Rightarrow \varprojlim_{\beta} H ^{u+r} ( \sQ^+ ( i, w, \beta ), \cO_{\sQ^+ ( i, w, \beta )} ( \lambda ) ).$$
Here $\varprojlim_{\beta} \R ^u (h_i)_* \cO_{\sQ^+ ( i, w, \beta )} ( \lambda )$ vanishes except for $u = 0$ by the induction hypothesis. Hence, we have
$$\R ^r ( \bar{q}_i )_* \left( \varprojlim_{\beta} (h_i)_* \cO_{\sQ^+ ( i, w, \beta )} ( \lambda ) \right) \cong \varprojlim_{\beta} H ^{r} ( \sQ^+ ( i, w, \beta ), \cO_{\sQ^+ ( i, w, \beta )} ( \lambda ) ) \hskip 5mm r \ge 0.$$

On the other hand, we have (the limit of) spectral sequences
$$\varprojlim_{\beta} \R ^u ( \overline{h}_i )_* ( \R ^r (q_i)_* \cO_{\sQ^+ ( i, w, \beta )} ( \lambda ) ) \Rightarrow \varprojlim_{\beta} H ^{u+r} ( \sQ^+ ( i, w, \beta ), \cO_{\sQ^+ ( i, w, \beta )} ( \lambda ) ),$$
that is convergent before taking $\varprojlim$. Combining the above yields the result.
\end{proof}

We return to the proof of Theorem \ref{main}. By Claim \ref{change}, we deduce a spectral sequence
\begin{equation}
H ^{t} ( \sQ ( s_i w ), \R^u (q_i)_* \cO^+_{\sQ^+ ( i, w )} ( \lambda ) ) \Rightarrow H ^{u+t} ( \P^1, \mathop{SL} ( 2 ) \times ^{B_i^0} W ( \lambda )_w ).
\end{equation}

Since the fiber of $q_i$ is contained in $\P^1$, it follows that $\R^k ( q_i )_* \cO_{\sQ^+ ( w, \beta )} = \{ 0 \}$ for $k \ge 2$. In addition, $\sQ^+ ( i, w, \beta )$ is contained in $\mathop{SL} ( 2 ) \times ^{B_i^0} \sQ^+ ( s_i w, \beta )$, where the natural prolongization of $q_i$ becomes a $\P^1$-fibration. Therefore, the short exact sequence
$$0 \rightarrow \ker \rightarrow \cO_{\mathop{SL} ( 2 ) \times ^{B_i^0} \sQ^+ ( s_i w, \beta )} \rightarrow \cO_{\sQ^+ ( i, w, \beta )} \rightarrow 0$$
yields a part of the long exact sequence
$$0 = \R^1 ( q_i )_* \cO_{\mathop{SL} ( 2 ) \times ^{B_i^0} \sQ^+ ( s_i w, \beta )} \rightarrow \R^1 ( q_i )_* \cO_{\sQ^+ ( i, w, \beta )} \rightarrow \R^2 (q_i)_* \ker = 0,$$
where the last equality follows by the relative dimension counting. Therefore, we conclude that
\begin{equation}
H ^k ( \sQ ( s_i w ), \cO^+_{\sQ ( s_i w )} ( \lambda ) ) \cong H ^{k} ( \P^1, \mathop{SL} ( 2 ) \times ^{B_i^0} W ( \lambda )_w ) \hskip 3mm \text{for each} \hskip 3mm k \in \Z,\label{2nd}
\end{equation}
where $\cO_{\sQ ( s_i w, \beta )}^+ ( \lambda ) := ( q_i )_* \cO_{\sQ^+ ( w, \beta )} ( \lambda )$. By construction, we have an embedding $\cO_{\sQ ( s_i w, \beta )} ( \lambda ) \hookrightarrow \cO_{\sQ ( s_i w, \beta )}^+ ( \lambda )$ (and we can take their inverse limits by construction). In particular, taking their global sections yield:
\begin{equation}
\mathrm{ch} \, H ^0 ( \sQ ( s_i w ), \cO_{\sQ ( s_i w )} ( \lambda ) ) ^* \le \mathrm{ch} \, H ^0 ( \sQ ( s_i w ), \cO_{\sQ ( s_i w )}^+ ( \lambda ) )^*.\label{norm}
\end{equation}

From (\ref{fund}), (\ref{norm}), and (\ref{2nd}), we deduce that
\begin{align}\nonumber
\mathrm{ch} \, W ( \lambda )_{s_i w} \le \mathrm{ch} \, H ^0 ( \sQ ( s_i w ), \cO_{\sQ ( s_i w )} ( \lambda ) ) ^* & \le \mathrm{ch} \, H ^0 ( \sQ ( s_i w ), \cO_{\sQ ( s_i w )}^+ ( \lambda ) )^*\\
& = \mathrm{ch} \, H ^{0} ( \P^1, \mathop{SL} ( 2 ) \times ^{B_i^0} W ( \lambda )_w )^*.\label{ch-comp}
\end{align}
Thanks to Corollary \ref{cNS} (and Theorem \ref{DCF}), we derive that all the inequalities in (\ref{ch-comp}) must be in fact an equality. In particular, this shows that all the sections of $\cO^+ _{\sQ ( s_i w , \beta )}( \lambda )$ and $\cO _{\sQ ( s_i w , \beta )}( \lambda )$ are the same by taking the inverse limit. A vector
$$f \in \mathrm{Im} \, ( \bigotimes _{j \in \mathtt I} W ( \varpi_j )^* \rightarrow W ( \rho )^*_{s_i w} ) \subset \mathrm{Im} \, ( R ( w_0 ) \rightarrow R ( s_i w ) )$$
defines a section of $\Gamma ( \sQ ( s_i w ), \cO_{\sQ ( s_i w )} ( \rho ) )$. Hence, it defines an inclusion $\cO _{\sQ (s_i w)} \hookrightarrow \cO _{\sQ (s_i w)} ( \rho )$ whose $n$-times repeated application gives $\cO _{\sQ ( s_i w )} \hookrightarrow \cO _{\sQ ( s_i w )} ( n \rho )$. This leads to a map $\cO^+ _{\sQ ( s_i w )} \hookrightarrow \cO _{\sQ ( s_i w )} ^+ ( n \rho )$. If $f$ is homogeneous of degree $\ge -m$, then it defines an affine open subspace  $\mathfrak U ( f, \beta )$ of each of $\sQ ( s_i w, \beta )$ for every $\beta \in Q_+^{\vee}$ so that 
\begin{equation}
\left< \beta, \varpi_i \right> \ge m \hskip 5mm \text{for each} \hskip 5mm i \in \mathtt I \label{deg-bd}
\end{equation}
by (\ref{Pemb}). Therefore, taking limit $n \to \infty$ is a localization to an affine open subset on $\sQ ( s_i w, \beta )$ whenever $\beta$ satisfies $(\ref{deg-bd})$. The affine schemes $\{\mathfrak U ( f, \beta )\}_{\beta}$ defines an ind-affine subset $\mathfrak U ( f ) : = \varinjlim_{\beta} \mathfrak U ( f, \beta )$. As the localization is flat, it commutes with $\varprojlim$ and $\Gamma$ as the condition (\ref{deg-bd}) is clearly satisfied for every $\beta' > \beta$ whenever $\beta$ satisfies $(\ref{deg-bd})$. Therefore, we conclude that
\begin{align*}
\Gamma ( \mathfrak U ( f ), \cO_{\sQ (s_i w)} ( n \rho ) ) & = \varinjlim_{n \to \infty} \Gamma ( \sQ ( s_i w ), \cO_{\sQ (s_i w)} ( n \rho ) ) \\
& = \varinjlim_{n \to \infty} \Gamma ( \sQ ( s_i w ), \cO^+_{\sQ (s_i w)} ( n \rho ) ) = \Gamma ( \mathfrak U ( f ), \cO^+_{\sQ (s_i w)} ( n \rho ) ).
\end{align*}
Since every further localization to a point of $\mathfrak U ( f )$ is realized as a projective system of local rings, we conclude that $\cO_{\sQ (s_i w)} \cong \cO^+_{\sQ (s_i w)}$ on $\mathfrak U ( f )$ (as pro-sheaves) again by the flatness of the localization. Here we have $\bigcap_f \mathfrak U ( f ) = \emptyset$ by (\ref{Pemb}) as every point of $\sQ ( s_i w )$ is a point of $\sQ ( s_i w, \beta )$ for some $\beta \in Q_+^{\vee}$. This shows that $\cO_{\sQ (s_i w)} \cong \cO^+_{\sQ (s_i w)}$ as pro-sheaves. Therefore, we conclude (\ref{1st}) (or the first assertion). The second assertion follows as $\eta$ must be an inclusion.

By the first assertion, $R ^{\#} ( w )$ is a dense subring of the projective coordinate ring of $\sQ ( w )$, and its graded completion is normal by Proposition \ref{inf-norm}. Therefore, Corollary \ref{homog2} implies the projective normality of $\sQ ( w )$.

This proceeds the induction on $w \in W$ with respect to $<_q$ from the base case $w = w_0$, and completes the proof of Theorem \ref{main}.
\end{proof}

\begin{thm}[Demazure character fomula for $\sQ ( w )$]\label{dmain}
For $\lambda \in \Lambda_+$, $\beta \in Q^{\vee}$, and $w, v \in W$ so that a reduced expression $s_{i_{l}} \cdots s_{i_1}$ of $t_{\beta} w$ yields a sequence
\begin{equation}
\overline{t_{\beta} w v} >_q \cdots >_q \overline{s_{i_2} s_{i_1} v} >_q \overline{s_{i_1} v} >_q v \hskip 3mm \text{s.t.} \hskip 3mm \ell ( t_{\beta} w v ) = l + \ell ( v )\label{length-add}
\end{equation}
$($for example, this happens when $\left< v^{-1}w^{-1} \beta, \alpha_i \right> \ge 0$ for every $i \in \mathtt I)$ we have
$$D_{t _{\beta} w} ( \mathrm{ch} \, \Gamma ( \sQ ( v ), \cO _{\sQ ( v )} ( \lambda ) )^* ) = q^{- \left< \beta, w v \lambda \right>} \cdot \mathrm{ch} \, \Gamma ( \sQ ( wv ), \cO _{\sQ ( wv )} ( \lambda ) )^*.$$
In particular, we have
$$D_w ( \mathrm{ch} \, W ( \lambda )_{v} ) = \mathrm{ch} \, W ( \lambda ) _{wv} \hskip 3mm \text{for} \hskip 3mm w,v \in W \text{ s.t. } \ell ( wv ) = \ell ( w ) + \ell ( v ).$$
\end{thm}

\begin{proof}
By the definition of $D_w$, it suffices to prove $\mathrm{ch} \, W ( \lambda ) _{s_{\vartheta} w} = q ^{- \left< \vartheta^{\vee}, w \lambda \right>} \cdot D_0 ( \mathrm{ch} \, W ( \lambda )_{w} )$ whenever $w^{-1} \vartheta \not\in \Delta^+$, and $\mathrm{ch} \, W ( \lambda ) _{s_{i} w} = D_i ( \mathrm{ch} \, W ( \lambda )_{w} )$ whenever $s_i w > w$ for $i \in \mathtt I$. We have
$$D_i ( \mathrm{ch} \, W ( \lambda )_{w} ) = \mathrm{ch} \, H ^0 ( {\mathbf I}_i \times ^{\mathbf I} \sQ (w ), \mathcal O _{\sQ ( w )} ( \lambda ) ) ^*$$
by Corollary \ref{cNS} and Theorem \ref{main} 1). Therefore, if we take into account the fact that the lowest degree term $v_{s_0 u \lambda}$ has degree $\left< \vartheta^{\vee}, u \lambda \right>$ for $u \in W$ when $\overline{s_0 u} >_q u$ and $\ell ( s_0 u ) = \ell ( u ) + 1$, then the result follows by induction with respect to (\ref{length-add}).
\end{proof}

\section{Feigin-Makedonskyi modules}

For each $\alpha \in \Delta$, we fix non-zero root vectors $e_{\alpha} \in \gu$ and $f_{\alpha} \in \gu^-$ of weight $\alpha$ and $- \alpha$, respectively. The following result is due to Feigin-Makedonskyi-Orr \cite{FMO} (see also Naito-Nomoto-Sagaki \cite{NNS} for its $t$-analogue), but we decided to include a proof as the author likes the following proof.

\begin{thm}\label{FM}
Let $\lambda \in \Lambda_+$ and $w \in W$. The module $W ( \lambda )_w$ is free over $\C [\A^{(\lambda)}]$ of rank $\dim \, W ( \lambda, 0 )$. In addition, the module $W ( \lambda )_w \otimes_{\C [\A^{(\lambda)}]} \C_0$ is generated by $v_{w\lambda}$ subject to the conditions:
\begin{itemize}
\item $( h \otimes z ) v_{w \lambda} = 0$ for every $h \in \h$;
\item In case $\alpha \in \Delta^+ \cap w \Delta ^+$, we have
$$e_{\alpha} v_{w \lambda} = 0 \hskip 3mm \text{and} \hskip 3mm ( f_{\alpha} \otimes z )^{\left< w ^{-1} \alpha^{\vee}, \lambda \right> + 1} v_{w \lambda} = 0;$$
\item In case $\alpha \in \Delta^+ \cap w \Delta ^-$, we have
$$( f _{\alpha} \otimes z ) v_{w \lambda} = 0 \hskip 3mm \text{and} \hskip 3mm e_{\alpha}^{- \left< w ^{-1} \alpha^{\vee}, \lambda \right> + 1} v_{w \lambda} = 0.$$
\end{itemize}
In other words, $W ( \lambda )_w \otimes_{\C [\A^{(\lambda)}]} \C_0$ is the generalized Weyl module $W _{w \lambda}$ in the sense of Feigin-Makedonskyi $\cite{FM}$.
\end{thm}

\begin{proof}
The $\C [\A^{(\lambda)}]$-action is realized by the $U ( z \h [z] )$-action on the highest weight vectors on $W ( \lambda )$, and hence so is for each extremal weight vector $v_{w \lambda}$. The other two conditions also hold for $v_{w \lambda} \in W ( \lambda )$ by examining possible $\h$-weights. As the both modules are cyclic, it follows that we have an $\gI$-module surjection
\begin{equation}
W _{w \lambda} \longrightarrow \!\!\!\!\! \rightarrow W ( \lambda )_w \otimes_{\C [\A^{(\lambda)}]} \C_0.\label{FMsurj}
\end{equation}
Since $W ( \lambda )_w$ contains some grading shift of $W ( \lambda )$ as its Demazure submodule, we conclude $W ( \lambda ) \subset W ( \lambda )_w \subset W ( \lambda )$ as $\C [\A^{(\lambda)}]$-modules. Here $W ( \lambda )$ is a free $\C [\A^{(\lambda)}]$-module of rank $\dim \, W ( \lambda, 0 )$ by Theorem \ref{free} 2). Therefore, we deduce that $W ( \lambda )_w$ is a torsion-free $\C [\A^{(\lambda)}]$-module of generic rank $\dim \, W ( \lambda, 0 )$. By the semicontinuity theorem, we have
$$\dim \, W _{w \lambda} = \dim \, W ( \lambda, 0 ) \le \dim \, W ( \lambda )_w \otimes_{\C [\A^{(\lambda)}]} \C_0,$$
where the first equality is Feigin-Makedonskyi \cite{FM} Theorem B. Therefore, (\ref{FMsurj}) forces that above inequality to be an equality. Again by (\ref{FMsurj}), we conclude $W _{w \lambda} \cong W ( \lambda )_w \otimes_{\C [\A^{(\lambda)}]} \C_0$. Moreover, this implies that $W ( \lambda )_w$ is $\C [\A^{(\lambda)}]$-free by (the graded version of) Nakayama's lemma.
\end{proof}

Let us set $\mathrm{Fl} ^{\frac{\infty}{2}} ( w )$ to be the quotient of $\sQ ( w )$ by the equivalence relation induced from the right $H [\![z]\!]_1$-action through the embedding $\sQ ( w ) \subset \mathbf Q_G ( w )$ (where $H [\![z]\!]_1 := \ker \, H [\![z]\!] \rightarrow H$).

\begin{cor}\label{BW}
We have an isomorphism
$$\Gamma ( \mathrm{Fl} ^{\frac{\infty}{2}} ( w ), \cO ( \lambda ) )^* \cong W _{w \lambda}.$$
\end{cor}

\begin{proof}
The space $\mathrm{Fl} ^{\frac{\infty}{2}} ( w )$ is the image of $\sQ ( w )$ in the free quotient of $\mathbf Q (w)$ by $H[\![z]\!]_1$. Hence, the infinitesimal version of our equivalence relation is realized by $z \h [z]$. It follows that its global section of a ($G[z]$-equivariant) line bundle is the $z \h [z]$-fixed part of that in $\sQ (w)$. Therefore, Theorem \ref{FM} implies the result by taking the $z \h [z]$-fixed part of $W ( \lambda )^*_w$.
\end{proof}

For each $\gamma \in \Lambda$, we have a polynomial $E_{\gamma} ( q, t ) \in \C (q,t) [\Lambda]$ defined in Cherednik \cite{C}. Let us define the bar involution on $\C (q,t) [\Lambda]$ as the ring involution so that $\overline{q^m t^n e^{\lambda}} := q^m t^n e^{- \lambda}$ for each $m, n \in Z$ and $\lambda \in \Lambda$. We set $E_{\gamma} ^{\dagger} ( q, t ) := \overline{E_{\gamma} ( q, t )}$.

\begin{thm}[\cite{FM} and \cite{Ion03, FL07, LNSSS2}]\label{nMac}
For $\lambda \in \Lambda_+$, we have
$$\mathrm{ch} \, W _{- w_0 \lambda} = E _{w_0 \lambda} ^{\dagger} ( q^{-1}, \infty), \hskip 3mm \text{and} \hskip 3mm \mathrm{ch} \, W _{- \lambda} = E _{w_0 \lambda} ^{\dagger} ( q, 0 ).$$ 
\end{thm}

\begin{proof}
The first equality is a consequence of Feigin-Makedonskyi \cite{FM}. The second equality is proved for type $\mathsf{ADE}$ as a combination of Ion \cite{Ion03} and Fourier-Littelmann \cite{FL07}, and in general by Lenart-Naito-Sagaki-Schilling-Shimozono \cite{LNSSS2} (cf. Chari-Ion \cite{CI14}).
\end{proof}

The first equality of the following assertion is \cite{CO} Proposition 2.5.

\begin{cor}\label{nMconn}
For $\lambda \in \Lambda_+$, we have equalities
\begin{align*}
D_{w_0} ( E_{w_0 \lambda} ^{\dagger} ( q^{-1}, \infty ) ) & = E_{w_0 \lambda}^{\dagger} ( q, 0 )\\
D_{w_0 t_{\beta}} ( E_{w_0 \lambda}^{\dagger} ( q, 0 ) ) & = q^{\left< \beta, \lambda \right>} \cdot E_{w_0 \lambda} ^{\dagger} ( q^{-1}, \infty ),
\end{align*}
where $\beta \in Q^{\vee}$ satisfies $\left< \beta, \alpha_i \right> < 0$ for each $i \in \mathtt I$.
\end{cor}

\begin{proof}
Taking Theorem \ref{nMac} into account, the both assertions follow directly by Theorem \ref{FM} and Theorem \ref{dmain}.
\end{proof}

\section{Non-symmetric Macdonald polynomials}

We keep the setting of the previous section. In this section, all cohomologies of (pro-)sheaves are graded $\gI$-modules obtained from some $\Gamma (\sQ ( w ), \cO ( \lambda ) )$ by a finite successive applications of $\h$-weight twists and taking cohomologies along $\P^1$ with making use of vector bundles $M \mapsto \mathop{SL} ( 2 ) \times^{B_i^0} M$. Moreover, such operations essentially deal with finitely many distinct $\h$-weights when we fix $\lambda \in \Lambda$. Therefore, Theorem \ref{main} and the fact that $\mathrm{ch} \, W ( \lambda )_{w}$ makes sense for each $w \in W$ guarantees the degree-wise Mittag-Leffler condition of the pro-systems defining our sheaves. To this end, we mostly drop the argument on the Mittag-Leffler conditions for the sake of simplicity.

Fix $v \in W$ and a sequence $\mathbf i = ( i_1, i_2,\ldots, i_{\ell})$ of elements of $\mathtt I$ of length $\ell$. We set $w \in W$ to be
\begin{equation}
w = s_{i_1} s_{i_2} \cdots s_{i_{\ell}}.\label{rexw}
\end{equation}
In case (\ref{rexw}) is a reduced expression of $w$, we say that $\mathbf i$ is a reduced expression of $w$. We call that $\mathbf i$ (or $w$) is adapted to $v$ if $\ell( wv ) = \ell + \ell ( v )$ (then $\mathbf i$ is a reduced expression of $w$).

We define
$$\sQ ( \mathbf i, v ) := \mathbf I _{i_1} \times ^{\mathbf I} \mathbf I_{i_2} \times^{\mathbf I} \cdots \times^{\mathbf I} \mathbf I_{i_{\ell}} \times^{\mathbf I} \sQ ( v ).$$
It induces the multiplication map
$$q _{\mathbf i, v} : \sQ ( \mathbf i, v ) \ni ( g_1,\ldots,g_{\ell}, x) \mapsto g_1 \cdots g_{\ell} x \in \sQ.$$

For each $1 \le j \le \ell$, we define a divisor $H _j \subset \sQ ( \mathbf i, v )$ as:
$$H _j = \{ ( g_1,\ldots,g_{\ell}, x) \in \sQ ( \mathbf i, v ) \mid g_j \in \mathbf I \subsetneq \mathbf I_{i_j} \}.$$

\begin{lem}\label{D-van0}
There exists $u \in W$ so that we have
$$\R^{k} ( q_{\mathbf i, v} ) _* \cO_{\sQ ( \mathbf i, v )} = \begin{cases} \cO _{\sQ ( u )} & (k = 0)\\
\{ 0 \} & (k \ge 1)\end{cases}.$$
\end{lem}

\begin{proof}
We first prove the case $\ell ( w ) = 1$. We set $\mathbf i = \{i\}$. In case $s_i v < v$, the space $\sQ ( i, v )$ is a $\P^1$-fibration over $\sQ ( i, v )$ through the map $q_{i, v}$ since $\mathbf I _i /\mathbf I \cong \P^1$. Hence, the assertion holds by setting $u = v$. We consider the case $s_i v > v$. By a similar argument as in Lemma \ref{1step}, we have a map
$$q_{i,s_i v} : \mathbf I _i \times^{\mathbf I} \sQ ( s_i v ) \longrightarrow \sQ ( s_i v ).$$
The map $q_{i,s_i v}$ is a $\P^1$-fibration. The fiber of $q_{i, v}$ along each point of $\sQ ( s_i v )$ is either $\mathrm{pt}$ or $\P^1$. By dimension estimate, we deduce that $\R^k ( q_{i,s_i v} ) _* \mathcal M = \{ 0 \}$ ($k \ge 2$) for every $\Gm$-equivariant pro-coherent sheaf on $\mathbf I _i \times^{\mathbf I} \sQ ( s_i v )$ satisfying the (degree-wise) Mittag-Leffler condition (or a $\Gm$-equivariant coherent sheaf on $P _i \times^{B} \sQ ( s_i v, \beta )$ for each $\beta \in Q_+^{\vee}$ when $i \in \mathtt I$). We have a short exact sequence
\begin{equation}
0 \rightarrow \ker \rightarrow \cO _{\mathbf I _i \times^{\mathbf I} \sQ ( s_i v )} \longrightarrow \cO _{\mathbf I _i \times^{\mathbf I} \sQ ( v )} \rightarrow 0
\end{equation}
that yields an exact sequence
$$\R^1 ( q_{i,s_i v} )_* \cO _{\mathbf I _i \times^{\mathbf I} \sQ ( s_i v )} \longrightarrow \R^1 ( q_{i,s_i v} )_* \cO _{\mathbf I _i \times^{\mathbf I} \sQ ( v )} \rightarrow \R^2 ( q_{i,s_i v} )_* \ker \equiv 0.$$
We have $\R^1 ( q_{i,s_i v} )_* \cO _{\mathbf I _i \times^{\mathbf I} \sQ ( s_i v )} = \{ 0 \}$ since $q_{i,s_iv}$ is a $\P^1$-fibration. Consequently, we have $\R^1 ( q_{i,s_i v} )_* \cO _{\mathbf I _i \times^{\mathbf I} \sQ ( v )} = \{ 0 \}$. Now the normality of $\sQ ( s_i w )$ implies $( q_{i,s_i v} )_* \cO _{\mathbf I _i \times^{\mathbf I} \sQ ( s_i v )} = ( q_{i,s_i v} )_* \cO _{\mathbf I _i \times^{\mathbf I} \sQ ( v )} = \cO _{\sQ ( s_i v )}$, that is the case of $\ell ( w ) = 1$ (cf. the proof of Theorem \ref{main}).

We assume the assertion holds for every pair $(\mathbf i, v)$ so that the length of $\mathbf i$ is $< \ell$ to proceed the induction. We set $\mathbf i' = \{i_2,i_3,\ldots,i_{\ell}\}$ and $v' = s_{i_1} u'$, where $u' \in W$ is obtained as $u$ in the assertion for $(\mathbf i', v)$. In case $v' > u'$, we have a factorization
\begin{equation}
\sQ ( \mathbf i, v ) \stackrel{q^1}{\longrightarrow} \sQ ( i_1, u' ) \stackrel{q^2}{\longrightarrow} \sQ ( v' )\label{nrc}
\end{equation}
so that $q_{\mathbf i, v} = q^2 \circ q^1$. The induction hypothesis yields $q^1_* \cO _{\sQ ( \mathbf i, v )} = \sQ ( i_1, u' )$ and $\R^k q^1_* \cO _{\sQ ( \mathbf i, v )} = \{ 0 \}$ for $k > 0$. In case $v' < u'$, we have a factorization map obtained from (\ref{nrc}) by replacing $v'$ with $u'$. Applying the case $\ell ( w ) = 1$, the induction (on $\ell$) proceeds in the both cases. Therefore, we conclude the assertion by induction. 
\end{proof}

In the following, we set $u ( \mathbf i, v )$ to be $u \in W$ determined by the pair $(\mathbf i, v)$ in Lemma \ref{D-van0}. For each $j \in [1,\ell]$, we set $\mathbf i_{j} \in \mathtt I^{\ell - 1}$ to be the sequence obtained by omitting the $j$-th entry, and we set $\mathbf i^j \in \mathtt I^{\ell - j}$ to be the sequence obtained by forgetting the first $j$ entries.

\begin{thm}[see \cite{BB05} Theorem 2.2.6]\label{cover}
For a fixed $ w \in W$, let $v \in W$ be a maximal element so that $v < w$. We have $\ell ( v ) = \ell ( w ) - 1$.
\end{thm}

\begin{prop}\label{strict}
Let $i \in \mathtt I$ and $e \neq w \in W$ so that $s_i w > w$. Let $S_1 \subset W$ be the set of maximal elements so that $v < w$ and $s_i v > v$, and let $S_2 \subset W$ be the set of maximal elements so that $v < w$ and $s_i v < v$. Then, we have
$$\sum_{v \in S_1 \cup S_2} W ( \lambda )_v = \left( \sum_{v \in S_1 \cup S_2} ( W ( \lambda )_{s_iv} + W ( \lambda )_{v} ) \right) \cap W ( \lambda )_{w} \subset W ( \lambda )_{s_i w}.$$
\end{prop}

\begin{proof}
By definition, the $\gI$-cyclic vector of $W ( \lambda )_v$ ($v \in S_1 \cup S_2$) belongs to $W ( \lambda )_{w}$, and one of $W ( \lambda )_{s_i v}$ ($v \in S_1$) or $W ( \lambda )_{v}$ ($v \in S_2$). Hence the inclusion $\subset$ is clear.

By the proof of Corollary \ref{cNS}, we have a uniform basis of $W ( \lambda)$ that spans $W ( \lambda )_{w}$ for each $w \in W$. As in the proof of Proposition \ref{inf-norm}, we define $W ( \lambda )_{ut_{\beta}}$ for $u \in W$ and $\beta \in Q^{\vee}_{+}$ by twisting the highest weight element of $W ( \lambda )$ by $q$-degree $\left< \beta, w_{0} \lambda \right>$. In view of Naito-Sagaki \cite{NS14} Theorem 4.2.1, we derive that the vector subspace
\begin{equation}
\left( \sum_{v \in S_1 \cup S_2} ( W ( \lambda )_{s_iv} + W ( \lambda )_{v} ) \right) \cap W ( \lambda )_{w} \subset W ( \lambda )\label{mint}
\end{equation}
is spanned by the sum of $W ( \lambda )_{ut_{\beta}}$ for $u \in W$ and $\beta \in Q^{\vee}_{+}$ so that $ut_{\beta}$ is smaller than both $s_{i} v$ (for some $v \in S_{1}$) and $w$ with respect to the semi-infinite Bruhat order (see e.g. \cite{LNSSS1} \S 6). Hence, it suffices to prove that an element of $W _{\mathrm{aff}}$ covered by both $w$ and $s_{i} v$ with respect to the semi-infinite Bruhat order for some $v \in S_{1}$ actually belongs to $W$.

For each $v \in S_{1}$, we have $v = s_{\beta} w$ for some $\beta \in \Delta^{+}$ by Theorem \ref{cover}. By Naito-Sagaki \cite{NS07} Lemma 2.11 and \cite{LNSSS1} Lemma 4.1, it suffices to prove that there exists no $\alpha, \gamma \in \Delta^{+}$ so that
\begin{align}\nonumber
\ell ( s_{\alpha} w ) - \ell ( w ) & = 2 \left< \rho, w^{-1} \alpha ^{\vee} \right> - 1 > 0\\\nonumber
\ell ( s_{\gamma} s_i v ) - \ell ( s_i v ) & = 2 \left< \rho, v^{-1} s_i \gamma ^{\vee} \right> - 1 > 0\\\label{transl}
w^{-1} \alpha ^{\vee} & = v^{-1} s_i \gamma ^{\vee}\\\label{ident}
s_\gamma s_i v & = s_\gamma s_i s_\beta w = s_\alpha w.
\end{align}
The condition (\ref{transl}) yields $\alpha ^{\vee} = s_{\beta} s_{i} \gamma ^{\vee}$, that is equivalent to $s_{\gamma} = s_{i} s_{\beta} s_{\alpha} s_{\beta} s_{i}$. 
Hence, the condition (\ref{ident}) forces $s_{i} s_{\beta} s_{\alpha} = s_{\alpha}$, that is equivalent to $s_{i} = s_{\beta}$. Since $v \in S_{1}$, we have $\beta \neq \alpha_{i}$. This is a contradiction, and hence an element of $W _{\mathrm{aff}}$ covered by both $w$ and $s_{i} v$ belongs to $W$. Therefore, we conclude that (\ref{mint}) is spanned by $W ( \lambda )_{u}$ for $w > u \in W$ as required.
\end{proof}

\begin{prop}\label{coh-tw}
For each $\lambda \in \Lambda$ and $I = [b+1,c] \subset [1,\ell]$ and the pair $(\mathbf i, v )$ so that $\ell ( u ( \mathbf i^{b}, v ) ) = |I| + \ell ( u ( \mathbf i^c, v ) )$, we have
$$H ^k ( \sQ ( \mathbf i, v), \cO _{\sQ ( \mathbf i, v )} ( \lambda - \sum_{j = b+1}^c H_j )) = \{ 0 \} \hskip 3mm k > 0.$$
\end{prop}

\begin{proof}
We decompose $\mathbf i$ into three pieces corresponding to $[1,b], (b,c], (c,\ell]$ as $\mathbf i^{-}, \mathbf i^{0}, \mathbf i^{+}$, respectively.

We set $H_{\mathbf i} := \sum_{1 \le j \le |\mathbf i|} H_j$. We first prove the cohomology vanishing assertion by induction on $|\mathbf i|$ when $|\mathbf i^-| = 0 = |\mathbf i^+|$ and $v = e$. The case $|\mathbf i|=0$ is Lemma \ref{D-van0} and Theorem \ref{main}. We have an equality
\begin{equation}
H^0 ( \sQ ( \mathbf i, e ), \cO _{\sQ ( \mathbf i, e )} ( \lambda - H_{\mathbf i} ) ) = \bigcap_{j \in I} H^0 ( \sQ ( \mathbf i, e ), \cO _{\sQ ( \mathbf i, e )} ( \lambda - H_j ) ).\label{intersect}
\end{equation}
For each $j \in I$, the short exact sequence of sheaves
$$0 \to \cO _{\sQ ( \mathbf i, e )} ( \lambda - H_j ) \to \cO _{\sQ ( \mathbf i, e )} ( \lambda ) \to \cO _{\sQ ( \mathbf i_j, e )} ( \lambda ) \to 0,$$
gives rise to a long exact seqeunce:
\begin{align*}
0 & \rightarrow H^0 ( \sQ ( \mathbf i, e ), \cO _{\sQ ( \mathbf i, e )} ( \lambda - H_j ) ) \rightarrow H^0 ( \sQ ( \mathbf i, e ), \cO _{\sQ ( \mathbf i, e )} ( \lambda ) ) \\
 & \rightarrow H^0 ( \sQ ( \mathbf i, v ), \cO _{\sQ ( \mathbf i_j, e )} ( \lambda ) ) \rightarrow H^1 ( \sQ ( \mathbf i, e ), \cO _{\sQ ( \mathbf i, e )} ( \lambda - H_j ) ) \rightarrow \cdots
\end{align*}
Applying Lemma \ref{D-van0} and Theorem \ref{main}, we conclude that
\begin{align}\label{H^0_c} H^0 ( \sQ ( \mathbf i, e ), \cO _{\sQ ( \mathbf i, e )} ( \lambda - H_j ) )^* & \cong W ( \lambda )_{u ( \mathbf i, e )} / W ( \lambda )_{u ( \mathbf i_j, e )}, \hskip 3mm \text{and} \\
H^{>0} ( \sQ ( \mathbf i, e ), \cO _{\sQ ( \mathbf i, e )} ( \lambda - H_j ) ) & = \{ 0\}\nonumber
\end{align}
for each $j \in I$. Note that the assertion follows from this when $\g$ is of type $\mathsf{A}_1$.

The induction hypothesis yields the desired cohomology vanishing if we replace $\mathbf i$ with $\mathbf i'$ obtained from $\mathbf i'$ by omitting the first entry $i_1$. In view of \cite{BB05} Corollary 2.2.3, each element of the set $S$ of maximal elements in $W$ that is smaller than $w$ is realized as $u ( \mathbf i'_j, v )$ for some $1 \le j \le \ell-1$. Since we have $W ( \lambda )_x \subset W ( \lambda )_y$ if $x < y \in W$, we can omit $j$ from the RHS of (\ref{intersect}) so that $u ( \mathbf i_j', v ) \not\in S$. By Proposition \ref{strict}, we deduce that
\begin{equation}
\sum_{x \in S} W ( \lambda )_x = (\sum_{x \in S} W ( \lambda )_x + W ( \lambda )_{s_{i_1} x} ) \cap W ( \lambda )_{u ( \mathbf i', e )} \subset W ( \lambda )_{u ( \mathbf i', e )} \subset W ( \lambda )_{u ( \mathbf i, e )}\label{incl}
\end{equation}
is the intersection of a graded $( H \cdot L_{i_1}^0)$-submodule of $W ( \lambda )_{u ( \mathbf i, e )}$ with $W ( \lambda )_{u ( \mathbf i', e )}$ (where $( H \cdot L_{i_1}^0)$ is the Levi subgroup of $\mathbf I _{i_1}$ that contains $H$). In view of Lemma \ref{d-est} (cf. the proof of Corollary \ref{cNS} and Theorem \ref{dmain}), the associated graded of (\ref{incl}) gives the direct sum of embeddings of $B_{i_1}^0$-modules of the form:
\begin{equation}
\{0\} \subset \{0\}, \{0\} \subset \C _{m \varpi}, \C_{m \varpi} \subset \C _{m\varpi}, \hskip 2mm \text{and} \hskip 2mm V ( m \varpi ) \subset V ( m\varpi ) \hskip 2mm  \text{inside} \hskip 2mm V ( m \varpi ),\label{class}
\end{equation}
where $V ( m \varpi )$ is the $(m+1)$-dimensional irreducible module of $L_{i_1}^0$, and $\C_{m \varpi}$ is its $B_{i_1}^0$-eigenspace. Since we have
$$H^k ( \sQ ( \mathbf i', e ), \cO_{\sQ ( \mathbf i', e )} ( \lambda - \sum H_{\mathbf i'} ) )^* \cong \begin{cases} W ( \lambda )_{u ( \mathbf i', e )} / \sum_{x \in S} W ( \lambda )_x & (k=0)\\ \{0 \} & (k > 0)\end{cases},$$
we deduce from (\ref{class}) with $\P^1$-calculations that
$$H^1 ( \sQ ( \mathbf i, e ), \cO_{\sQ ( \mathbf i, e )} ( \lambda - H_{\mathbf i} ) ) \cong H^1 ( \P^1, \mathcal F (-1) ) = \{ 0 \},$$
where $\mathcal F$ is the $( H \cdot L_{i_1}^0 )$-equivariant vector bundle induced from the graded $( H \cdot B_{i_1}^0 )$-module $W ( \lambda )_{u ( \mathbf i', e )} / \sum_{x \in S} W ( \lambda )_x$. Therefore, the long exact sequence associated to the short exact sequence
$$0 \to \cO_{\sQ ( \mathbf i, e )} ( \lambda - H_{\mathbf i} ) \to \cO_{\sQ ( \mathbf i, e )} ( \lambda - H_{\mathbf i'} ) \to \cO_{\sQ ( \mathbf i', e )} ( \lambda - H_{\mathbf i'} ) \to 0$$
(where $H_{\mathbf i'}$ on $\sQ ( \mathbf i, e )$ is the inflation of $H_{\mathbf i'}$ on $\sQ ( \mathbf i', e )$), together with the induction hypothesis implies the desired cohomology vanishing. This proceeds the induction, and consequently we have obtained the assertion when $|\mathbf i^-| = 0 = |\mathbf i^+|$ and $v = e$.

Next, we prove the assertion when $|\mathbf i^+| = 0$ and $v = e$. In view of the above discussion using (\ref{intersect}), Lemma \ref{d-est}, and Corollary \ref{cNS}, adding one element in the beginning of $\mathbf i^-$ amounts to form the corresponding vector bundles and then take its cohomology. In this case, the associated graded of an analogous filtration on
$$\sum_{b < j \le c} W ( \lambda )_{u ( \mathbf i'_j, e )} \subset W ( \lambda )_{u ( \mathbf i', e )} \subset W ( \lambda )_{u ( \mathbf i, e )}$$
adds some more cases $( \C_{m \varpi} = ) V \subsetneq V ( m \varpi ) \subset V ( m \varpi )$ to (\ref{class}). This is harmless (as we have no degree $-1$ twist) in this case, and we conclude the result by induction when $|\mathbf i^+| = 0$ and $v = e$.

We prove the assertion when $v = e$ and $\ell ( u ( \mathbf i^b, e )) = \ell - b$ hold using the previously shown cases. We consider the short exact sequence
$$0 \rightarrow \cO_{\sQ ( \mathbf j', e )} ( \lambda - \sum_{j = b+1}^{\ell - m} H_j) \rightarrow \cO_{\sQ ( \mathbf j', e )} ( \lambda - \sum_{j = b+1}^{\ell - m-1} H_j) \rightarrow \cO_{\sQ ( \mathbf j'', e )} ( \lambda - \sum_{j = b+1}^{\ell - m-1} H_j) \rightarrow 0,$$
where $\mathbf j'$ is the sequence formed by the first $(\ell - m)$ letters in $\mathbf i$, and $\mathbf j''$ is formed by the first $(\ell - m-1)$ letters in $\mathbf i$. By arguing from the case $m=0$ (corresponding to $|\mathbf i^{+}| = 0$), a repeated use of long exact sequences implies the result when $v = e$ and $\ell ( u ( \mathbf i^b, e )) = \ell - b$ hold by induction.

We can factor $q_{\mathbf i, e}$ into the composition of $q_{\mathbf i^-, \mathbf i^0, u ( \mathbf i^+, e)}$ and the inflation of $q_{\mathbf i^+, e}$. Hence, the Leray spectral sequence gives the general case of the assertion from Lemma \ref{D-van0} and the previously known cases as required.
\end{proof}

\begin{cor}\label{wD-van}
For each $\lambda \in \Lambda$ and $I = ( b,c ] \subset [1,\ell]$ so that $\ell ( u ( \mathbf i^b, v ) ) = |I| + \ell ( u ( \mathbf i^c, v ) )$, we have
$$\R^{k} ( q_{\mathbf i, v} ) _* \cO_{\sQ ( \mathbf i, v )} ( - \sum_{j \in I} H_j ) = \{ 0 \} \hskip 3mm \text{for each} \hskip 2mm k > 0.$$
\end{cor}

\begin{proof}
In view of Corollary \ref{amplecone}, the assertion follows from the projection formula and Proposition \ref{coh-tw} by the Leray spectral sequence.\end{proof}

In the following, we assume that $\mathbf i$ is a reduced expression of $w$ unless stated otherwise. Note that the assumption of Corollary \ref{wD-van} holds automatically. For each $\lambda \in \Lambda_+$, we set
$$\mathcal E _{w} ( \lambda ) := ( q_{\mathbf i, e} )_* \cO_{\sQ ( \mathbf i, e )} ( \lambda - \sum_{k = 1}^{\ell} H_k ) \cong \left( ( q_{\mathbf i, e} )_* \cO_{\sQ ( \mathbf i, e )} ( - \sum_{k = 1}^{\ell} H_k ) \right) \otimes_{\cO _{\sQ}} \cO _{\sQ} ( \lambda ).$$
We have a natural inclusion $\mathcal E _{w} ( \lambda ) \subset \cO _{\sQ ( w )} ( \lambda )$ defined as:
\begin{equation}
\mathcal E _{w} ( \lambda ) \equiv ( q_{\mathbf i, e} )_* \cO_{\sQ ( \mathbf i, e )} ( \lambda - \sum_{k = 1}^{\ell} H_k ) \hookrightarrow ( q_{\mathbf i, e} )_* \cO_{\sQ ( \mathbf i, e )} ( \lambda ) \equiv \cO_{\sQ ( w )} ( \lambda ).\label{inclE}
\end{equation}
By construction, each $\mathcal E _{w} ( \lambda )$ is $( \mathbf I \rtimes \Gm )$-equivariant.

\begin{lem}
For each $w \in W$ and $\lambda \in \Lambda_+$, the sheaf $\mathcal E_w ( \lambda )$ is independent of the choice of a reduced expression of $w$.
\end{lem}

\begin{proof}
It is enough to check that all reduced expressions of $w$ gives rise to the same sheaf. We borrow notation from the proof of Theorem \ref{coh-tw} by choosing $v = e$ and $w \in W$. Then, (\ref{H^0_c}) and (\ref{intersect}) imply
$$H ^0 ( \sQ ( w ), \mathcal E _{w} ( \lambda ) )^* \cong W ( \lambda ) _{w} / \sum_{j = 1}^{\ell} W ( \lambda ) _{u ( \mathbf i_j, e )},$$
where $\mathbf i \in \mathtt I^{\ell}$ is a reduced expression of $w$. We have $W ( \lambda ) _{x} \subset W ( \lambda ) _{y}$ when $x < y$. In view of \cite{BB05} Corollary 2.2.3, we know that $u ( \mathbf i_j, e ) < w$, and $\{ u ( \mathbf i_j, e ) \}_{j = 1}^{\ell} \subset W$ exhausts the set of maximal elements in $W$ so that $< w$. In particular, the vector space $H ^0 ( \sQ ( w ), \mathcal E _{w} ( \lambda ) )$ depends only on $w$, and is independent of the choice of $\mathbf i$. Therefore, Corollary \ref{amplecone} forces all choices of $\mathbf i$ in the construction of the sheaves $\mathcal E _{w} ( \lambda )$ define the same subsheaf of $\cO_{\sQ ( w )} ( \lambda )$ as required.
\end{proof}

For $\lambda \in \Lambda_+$, we set $W _{\lambda} := \left< s_{i} \mid \left< \alpha^{\vee}_i, \lambda \right> = 0 \right>$. We define $W^{\lambda}$ to be the set of minimal length representatives of the coset $W / W_{\lambda}$.

\begin{lem}\label{restE}
For each $\lambda \in \Lambda_+$ and $w \in W^{\lambda}$, the module $\Gamma ( \sQ ( w ), \mathcal E_w ( \lambda ) )^*$ has an $\gI$-cyclic vector with its $\h$-weight $w \lambda$.
\end{lem}

\begin{proof}
By construction, we have an inclusion $\mathcal E_w ( \lambda ) \subset \cO _{\sQ ( w )} ( \lambda )$. This results a surjection $W ( \lambda )_w \rightarrow \Gamma ( \sQ ( w ), \mathcal E_w ( \lambda ) )^*$ of $\gI$-modules by taking the dual of their global sections. Since $W ( \lambda )_w$ is an $\gI$-module with a cyclic vector of weight $w \lambda$, so is $\Gamma ( \sQ ( w ), \mathcal E_w ( \lambda ) )^*$.
\end{proof}

\begin{thm}[Cherednik-Orr \cite{CO} Proposition 2.5]\label{COr}
Let $\lambda \in \Lambda_+$ and let $w \in W^{\lambda}$ so that $s_i w > w$ and $s_i w \in W^{\lambda}$ for some $i \in \mathtt I$. Then, we have:
\begin{enumerate}
\item If $w^{-1} \alpha_i = \alpha_j$ for some $j \in \mathtt I$ so that $\left< \alpha_j^{\vee}, \lambda \right> > 0$, then we have
$$( 1 - q^{\left< \alpha_j^{\vee}, \lambda \right>} ) E_{-s_i w \lambda} ^{\dagger} ( q^{-1}, \infty ) = D_i \left( E_{-w \lambda} ^{\dagger} ( q^{-1}, \infty ) \right) - E_{-w \lambda} ^{\dagger} ( q^{-1}, \infty );$$
\item If $w^{-1} \alpha_i$ is not a simple root, then we have
$$E_{- s_i w \lambda} ^{\dagger} ( q^{-1}, \infty ) = D_i \left( E_{- w \lambda} ^{\dagger} ( q^{-1}, \infty ) \right) - E_{- w \lambda} ^{\dagger} ( q^{-1}, \infty ).$$
\end{enumerate}
\end{thm}

\begin{proof}
If we set $T_i := D_i - 1$, then the adjoint of the bar-involution yields the Hecke operator $T_i$ specialized to $t = \infty$ (see e.g. \cite{NNS} $1^{\text{st}}$ ver. 5.1). Therefore, the current formulation is equivalent to \cite{CO} Proposition 2.5.
\end{proof}

\begin{prop}\label{epMac}
For each $\lambda \in \Lambda_+$ and $w \in W^{\lambda}$, we define $\lambda _w := \lambda - \sum_{w \alpha_j < 0} \varpi_j$. Then, we have
$$\sum_{i \ge 0} (-1)^i \mathrm{ch} \, H^i ( \sQ ( w ), \mathcal E_{w} ( \lambda ) ) ^* = \left( \prod_{i \in \mathtt I} \prod_{k = 1} ^{\left< \alpha_i ^{\vee}, \lambda_w \right>} \frac{1}{1 - q^{k}} \right) \cdot E_{- w \lambda} ^{\dagger} ( q^{-1}, \infty ).$$
\end{prop}

\begin{proof}
We define the (dual) Euler-Poincar\'e characteristic of an $( \mathbf I \rtimes \Gm )$-equivariant (pro-)coherent sheaf $\mathcal F$ by
$$\chi ( \mathcal F ) := \sum_{i \ge 0}  ( -1 )^i \mathrm{ch} \, H^i ( \sQ, \mathcal F ) ^* \in \Q (\!(q)\!) [\Lambda] \cup \{ \infty \},$$
where we understand it to be $\infty$ if one the coefficient of a monomial is $\infty$.

We prove the assertion by induction on $w \in W^{\lambda}$ (as every $w \in W^{\lambda}$ is connected to $e$ by the left multiplications of $\{ s_i \}_{i \in \mathtt I}$ inside $W^{\lambda}$). The case $w = e$ is Theorem \ref{nMac}. Hence, we assume the assertion for every $v < w$ to deduce the assertion for $w$. For $i \in \mathtt I$ so that $w < s_i w \in W^{\lambda}$, we set $H := H_1$ and $q_i := q_{i,w}$ for simplicity. By Corollary \ref{wD-van}, we have a short exact sequence
$$0 \rightarrow (q_i)_* \mathcal E_{w}^+ ( - H ) \rightarrow (q_i)_* \mathcal E^+_{w} \rightarrow \mathcal E_{w} \rightarrow 0,$$
where we denote $\mathcal E^+_{w}$ the inflation of $\mathcal E_w$ from $\sQ ( w )$ to $\sQ ( i, w )$. Now we have
\begin{equation}
\chi ( (q_i)_* \mathcal E_{w}^+ ( - H ) ) = D_i ( \chi ( \mathcal E_w ) ) - \chi ( \mathcal E_w ).\label{E-exact}
\end{equation}

In case $w^{-1} \alpha_i = \alpha_j$ for some $j \in \mathtt I$, then we have $\lambda_{s_i w} = \lambda_w - \varpi_j$. Therefore, the comparison of (\ref{E-exact}) and Theorem \ref{COr} 1) proceeds the induction.

In case $w^{-1} \alpha_i \not\in \Pi$, then we have $\lambda_{s_i w} = \lambda_w$. Therefore, the comparison of (\ref{E-exact}) and Theorem \ref{COr} 2) proceeds the induction.

These proceed the induction in both cases as required.
\end{proof}

\begin{cor}\label{gnsMac}
For each $\lambda \in \Lambda_+$ and $w \in W^{\lambda}$, we define $\lambda _w := \lambda - \sum_{w \alpha_j < 0} \varpi_j$. Then, we have
$$\mathrm{ch} \, H^i ( \sQ ( w ), \mathcal E_{w} ( \lambda ) ) ^* = \begin{cases}
\left( \prod_{i \in \mathtt I} \prod_{k = 1} ^{\left< \alpha_i ^{\vee}, \lambda_w \right>} \frac{1}{1 - q^{k}} \right) \cdot E_{- w \lambda} ^{\dagger} ( q^{-1}, \infty ) & (i=0)\\
0 & ( i > 0)
\end{cases}.$$
\end{cor}

\begin{proof}
By setting $\mathbf i$ to be adapted to $e$, Proposition \ref{coh-tw} and Corollary \ref{wD-van} implies
$$H^i ( \sQ ( w ), \mathcal E_{w} ( \lambda ) ) ^* = \{ 0 \} \hskip 3mm i > 0$$
Therefore, Proposition \ref{epMac} yields the result.
\end{proof}

{\footnotesize
\bibliography{coref}
\bibliographystyle{plain}}
\end{document}